\documentclass[11pt]{amsart}
\usepackage{latexsym,amscd,amssymb, graphicx, color, amsthm, bm, amsmath, cancel, enumitem}  
\usepackage{tikz}
\usepackage{tikz-3dplot}
\usepackage[margin=1in]{geometry}
\usepackage{hyperref}
\usepackage[all,cmtip]{xy}

\numberwithin{equation}{section}

\newtheorem{theorem}{Theorem}[section]
\newtheorem{proposition}[theorem]{Proposition}
\newtheorem{corollary}[theorem]{Corollary}
\newtheorem{lemma}[theorem]{Lemma}

\newtheorem{observation}[theorem]{Observation}

\newtheorem{problem}[theorem]{Problem}
\newtheorem{example}[theorem]{Example}

\theoremstyle{definition}
\newtheorem{defn}[theorem]{Definition}

\newcommand{\Stir}{{\mathrm{Stir}}}

\newcommand{\symm}{{\mathfrak{S}}}
\newcommand{\II}{{\mathbf{I}}}
\newcommand{\gr}{{\mathrm {gr}}}
\newcommand{\BBB}{{\mathcal{B}}}

\newcommand{\MMM}{{\mathcal{M}}}

\newcommand{\EEE}{{\mathcal{E}}}
\newcommand{\PPP}{{\mathcal{P}}}

\newcommand{\AAA}{{\mathcal{A}}}

\newcommand{\CCC}{{\mathcal{C}}}

\newcommand{\ZZZ}{{\mathcal{Z}}}

\newcommand{\KKK}{{\mathcal{K}}}
\newcommand{\CC}{{\mathbb{C}}}
\newcommand{\QQ}{{\mathbb{Q}}}
\newcommand{\ZZ}{{\mathbb{Z}}}
\newcommand{\PP}{{\mathbb{P}}}

\newcommand{\WWW}{{\mathcal{W}}}
\newcommand{\UUU}{{\mathcal{U}}}

\newcommand{\Zpoints}{{\mathcal{Z}}}

\newcommand{\LLL}{\mathcal{L}}

\newcommand{\xxx}{{\mathbf{x}}}
\newcommand{\yyy}{{\mathbf{y}}}
\newcommand{\ttt}{{\mathbf{t}}}
\newcommand{\zzz}{{\mathbf{z}}}
\newcommand{\ccc}{{\mathbf{c}}}
\newcommand{\eee}{{\mathbf{e}}}

\newcommand{\aalpha}{{{\bm \alpha}}}

\newcommand{\im}{\text{im}}
\newcommand{\conv}{\text{conv}}

\newcommand{\supp}{\mathrm{supp}}

\newcommand{\initial}{\text{in}}

\newcommand{\Gr}{\mathrm{Gr}}

\newcommand{\PM}{\mathrm{PM}}
\newcommand{\std}{\text{std}}

\begin{document}

\title[Equivariant cohomology of Grassmannian spanning lines]
{Equivariant cohomology of Grassmannian spanning lines}

\author[Raymond Chou]{Raymond Chou}
\email{r2chou@ucsd.edu}

\author[Tomoo Matsumura]{Tomoo Matsumura}
\email{matsumura.tomoo@icu.ac.jp}

\author[Brendon Rhoades]{Brendon Rhoades}
\email{bprhoades@ucsd.edu}

\begin{abstract}
Given integers $n \geq k \geq d$, let $X_{n,k,d}$ be the moduli space of $n$-tuples of lines $(\ell_1, \dots, \ell_n)$ in $\CC^k$ such that $\ell_1 + \cdots + \ell_n$ has dimension $d$. We give a quotient presentation of the torus-equivariant cohomology of $X_{n,k,d}$. The form of this presentation, and in particular the torus parameters appearing therein, will arise from the orbit harmonics method of combinatorial deformation theory. 
\end{abstract}

\maketitle

\section{Introduction}
\label{sec:Introduction}

Let $X$ be a space carrying an action of the rank $k$ torus $T = (\CC^*)^k$. The {\em equivariant cohomology ring} $H^*_T(X;\ZZ)$ is an enhancement of the usual singular cohomology of $X$ which carries information about the $T$-action. If $X = \{\bullet\}$ is a single point, we have a natural identification
\begin{equation}
    H^*_T(\{\bullet\};\ZZ) = \ZZ[\ttt_k] := \ZZ[t_1, \dots, t_k]
\end{equation}
with the polynomial ring in $k$ variables.
The map $X \to \{\bullet\}$ to the one-point space equips $H^*_T(X;\ZZ)$ with the structure of a $\ZZ[\ttt_k]$-module. 
Computing and understanding  equivariant cohomology rings is a major theme in Schubert calculus \cite{AHMMS, GKM, GZ, KM, SW, Tymoczko}. 

In this paper we use a deformation technique called {\em orbit harmonics} as a tool for computing equivariant cohomology. We give a quotient presentation of the equivariant cohomology of the following moduli space. A {\em line} in a vector space $V$ is a one-dimensional subspace $\ell \subseteq V$.

\begin{defn}
    \label{def:x-space}
    For positive integers $n \geq k \geq d$, let $X_{n,k,d}$ be the  space of $n$-tuples $\ell_\bullet = (\ell_1, \dots, \ell_n)$ of lines in $\CC^k$ such that the sum $\ell_1 + \cdots + \ell_n$ has dimension $d$ as a vector subspace of $\CC^k$.
\end{defn}

Writing $\PP^{k-1}$ for the complex projective space of lines in $\CC^k$, the space $X_{n,k,d}$ is a locally closed subvariety of $(\PP^{k-1})^n$. The natural action of $T = (\CC^*)^k$ on  $\PP^{k-1}$ induces a $T$-action on  $(\PP^{k-1})^n$ via $$t \cdot (\ell_1, \dots, \ell_n) := (t \cdot \ell_1, \dots, t \cdot \ell_n).$$ The subspace $X_{n,k,d}$ is stable under this $T$-action, so we have an equivariant cohomology ring $H^*_T(X_{n,k,d};\ZZ)$.

When $n = k = d$, the space $X_{n,n,n}$ is homotopy equivalent to the classical variety of complete flags $V_\bullet = (V_0 \subset V_1 \subset \cdots \subset V_n)$ in $\CC^n$. When $k = d$, the variety $X_{n,k} := X_{n,k,k}$ was introduced by Pawlowski and Rhoades \cite{PR} to give a geometric model for the Haglund-Rhoades-Shimozono generalized coinvariant rings $R_{n,k}$ \cite{HRS}. Computing the equivariant cohomology of $X_{n,k}$ was left as an open problem  \cite[Prob. 9.8]{PR}; we solve this problem here (Corollary~\ref{cor:pr-problem}).

In order to state our main result right away, we need some notation. Let $\Gr(d,\CC^k)$ be the Grassmannian of $d$-dimensional subspaces of $\CC^k$. We have a natural surjection
\begin{equation}
    p: X_{n,k,d} \twoheadrightarrow \Gr(d,\CC^k)
\end{equation}
which sends an $n$-tuple $\ell_\bullet = (\ell_1,\dots,\ell_n)$ of lines to the vector space sum $p(\ell_\bullet) = \ell_1 + \cdots + \ell_n$. The natural action of $T = (\CC^*)^k$ on $\CC^k$ induces an action on $\Gr(d,\CC^k)$, and the equivariant cohomology ring $H^*_T(\Gr(d,\CC^k),\ZZ)$ is generated by the equivariant Chern classes of the rank $d$ tautological vector bundle $\UUU_d$ over $\Gr(d,\CC^k)$. 

In addition to the torus variables $\ttt_k = (t_1, \dots, t_k)$, we let $\xxx_n = (x_1, \dots, x_n)$ be a length $n$ list of $x$-variables and $\yyy_d = (y_1, \dots, y_d)$ be a length $d$ list of $y$-variables. We write $\ZZ[\xxx_n,\yyy_d,\ttt_k]$ for the polynomial ring over these variables with integer coefficients. The symmetric group $\symm_d$ acts on the $y$-variables in $\ZZ[\xxx_n,\yyy_d,\ttt_k]$. Write $\ZZ[\xxx_n,\yyy_d,\ttt_k]^{\symm_d}$ for the subring of $\symm_d$-invariants. Let $e_r$ and $h_r$ be the elementary and complete homogeneous symmetric polynomials of degree $r$.

\begin{theorem}
    \label{thm:borel-presentation}
    For positive integers $n \geq k \geq d$, let $I_{n,k,d} \subseteq \ZZ[\xxx_n,\yyy_d,\ttt_k]^{\symm_d}$ be the ideal generated by 
    \begin{itemize}
        \item $e_r(\ttt_k) - e_{r-1}(\ttt_k) h_1(\yyy_d) + \cdots + (-1)^r h_r(\yyy_d)$ for $r > k-d$, 
        \item $e_r(\xxx_n) - e_{r-1}(\xxx_n) h_1(\yyy_d) + \cdots + (-1)^r h_r(\yyy_d)$ for $r > n-d$, and
        \item $x_i^d - x_i^{d-1} e_1(\yyy_d) + \cdots + (-1)^d e_d(\yyy_d)$ for $i = 1, \dots, n$.
    \end{itemize}
    If $T = (\CC^*)^k$, the $T$-equivariant cohomology of $X_{n,k,d}$ has presentation
    \begin{equation}
        H^*_T(X_{n,k,d};\ZZ) = \ZZ[\xxx_n,\yyy_d,\ttt_k]^{\symm_d}/I_{n,k,d}
    \end{equation}
    where 
    \begin{itemize}
        \item $x_i$ represents the equivariant Chern class of the line bundle $\LLL_i$ over $X_{n,k,d}$ with fiber $\ell_i$ over $\ell_\bullet = (\ell_1, \dots, \ell_n)$,
        \item the variables $y_1, \dots, y_d$ represent equivariant Chern roots of the rank $d$ tautological vector bundle $\UUU_d$ over $\Gr(d,\CC^k)$ pulled back along $p: X_{n,k,d} \to \Gr(d,\CC^k)$, and
        \item $t_i \in H^2_T(X_{n,k,d};\ZZ)$ represents the equivariant Chern class of the tautological line bundle over the $i^{th}$ factor of $BT = (\PP^{\infty})^k$.
    \end{itemize}
\end{theorem}

The geometric analysis used in the proof of Theorem~\ref{thm:borel-presentation} involves applying the Leray-Hirsch Theorem to the fiber bundle $p: X_{n,k,d} \to \Gr(d,\CC^k)$, invoking affine paving results appearing in \cite{PR}, and using Whitney sum reasoning to deduce that the relations of $I_{n,k,d}$ hold in the ring $H^*_T(X_{n,k,d};\ZZ)$; see Section~\ref{sec:Cohomology} for details. 

The algebraic analysis of the quotient ring in Theorem~\ref{thm:borel-presentation} uses ideas which we hope will be helpful in obtaining quotient presentations of other equivariant cohomology rings. Let $\Zpoints \subseteq \QQ^n$ be a finite locus of points in rational affine $n$-space $\QQ^n$. We have the vanishing ideal
\begin{equation}
    \II(\Zpoints) := \{ f \in \QQ[\xxx_n] \,:\, f(\zzz) = 0 \text{ for all $\zzz \in \Zpoints$} \} \subseteq \QQ[\xxx_n]
\end{equation}
of the locus $\Zpoints$. The {\em orbit harmonics} method associates to $\Zpoints$ the homogeneous quotient ring $\QQ[\xxx_n]/ \gr \, \II(\Zpoints)$ where $\gr \, \II(\Zpoints) \subseteq \QQ[\xxx_n]$ is the associated graded ideal of $\II(\Zpoints)$. Geometrically, this corresponds to a flat deformation of the locus $\Zpoints$ to a subscheme of $\QQ^n$ supported at the origin (which will be nonreduced whenever $|\Zpoints| > 1$). We have an isomorphism of vector spaces 
\begin{equation}
    \label{eq:orbit-harmonics-isomorphism}
    \QQ[\Zpoints] = \QQ[\xxx_n]/\II(\Zpoints) \cong \QQ[\xxx_n]/\gr \, \II(\Zpoints)
\end{equation}
where $\QQ[\xxx_n]/\gr \, \II(\Zpoints)$ is a graded vector space. If the locus $\Zpoints$ is a stable under the action of a finite subgroup $G \subseteq GL_n(\QQ)$, \eqref{eq:orbit-harmonics-isomorphism} is an isomorphism of $G$-modules, where $\QQ[\xxx_n]/\gr \, \II(\Zpoints)$ is a graded $G$-module.
The orbit harmonics deformation is shown schematically below in the case of a locus of size $|\Zpoints| = 6$ in $\QQ^2$ carrying an action of $G \cong \symm_3$ via reflection in the three displayed lines.

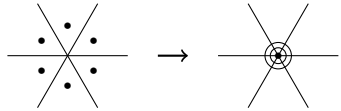
\begin{center}
 \begin{tikzpicture}[scale = 0.2]
\draw (-4,0) -- (4,0);
\draw (-2,-3.46) -- (2,3.46);
\draw (-2,3.46) -- (2,-3.46);

 \fontsize{5pt}{5pt} \selectfont
\node at (0,2) {$\bullet$};
\node at (0,-2) {$\bullet$};

\node at (-1.73,1) {$\bullet$};
\node at (-1.73,-1) {$\bullet$};
\node at (1.73,-1) {$\bullet$};
\node at (1.73,1) {$\bullet$};

\draw[thick, ->] (6,0) -- (8,0);

\draw (10,0) -- (18,0);
\draw (12,-3.46) -- (16,3.46);
\draw (12,3.46) -- (16,-3.46);

\draw (14,0) circle (15pt);
\draw(14,0) circle (25pt);
\node at (14,0) {$\bullet$};

 \end{tikzpicture}
\end{center}

Let $X$ be a variety carrying an action of $T = (\CC^*)^k$. In a number of settings, the ordinary cohomology $H^*(X;\QQ)$ has an orbit harmonics interpretation as follows. There exists a non-empty Zariski-open set $U \subseteq \QQ^k$ such that
\begin{itemize}
    \item for any $\aalpha_k = (\alpha_1, \dots, \alpha_k) \in U$ we have a finite locus $\Zpoints(\aalpha_k) \subseteq \QQ^n$,
    \item the homogeneous ideal $\gr \, \II(\Zpoints(\aalpha_k)) \subseteq \QQ[\xxx_n]$ does not depend on $\aalpha_k \in U$, and
    \item for any $\aalpha_k \in U$, we have the presentation $H^*(X;\QQ) = \QQ[\xxx_n]/ \gr \, \II(\Zpoints(\aalpha_k))$.
\end{itemize}
One should think of $\QQ^k$ as a rational form $\mathfrak{t}_\QQ$ of the Lie algebra $\mathfrak{t} \cong \CC^k$ of $T$. In the type A setting of this paper, the Zariski-open set $U$ will always be the set of points $\aalpha_k = (\alpha_1, \dots, \alpha_k) \in \QQ^k$ with distinct coordinates.

The cohomology rings of (generalized) Springer fibers \cite{GP, Griffin, GLW} and spaces of spanning subspace configurations \cite{HRS, PR, RhoadesSpanning, RW} have orbit harmonics interpertations as in the last paragraph.
With an eye towards equivariant cohomology,  we put the loci $\Zpoints(\aalpha_k)$ into a family by letting $\Zpoints \subseteq \QQ^{n+k}$ be the variety
\begin{equation}
    \Zpoints = \text{Zariski closure of } \bigcup_{\aalpha_k \, \in \, U} \Zpoints(\aalpha_k) \times \{ \aalpha_k \}
\end{equation}
Writing $\QQ[\xxx_n,\ttt_k]$ for the coordinate ring of $\QQ^{n+k}$, we have the vanishing ideal $\II(\Zpoints) \subseteq \QQ[\xxx_n,\ttt_k]$.  
The approach of this paper is to {\em predict} that the equivariant cohomology of $X$ has quotient presentation
\begin{equation}
    H^*_T(X;\QQ) = \QQ[\xxx_n,\ttt_k]/\II(\Zpoints)
\end{equation}
where the $t_i$ come from the torus action. Theorem~\ref{thm:borel-presentation} establishes this prediction when $X = X_{n,k,d}$; further algebraic arguments allow for the passage from rational to integral cohomology. Assuming the action of $T$ on $X$ is equivariantly formal, ordinary and equivariant cohomology are related by 
\begin{equation}
    \label{eq:ordinary-equivariant-relationship}
    H^*(X;\QQ) = \QQ \otimes_{\QQ[\ttt_k]} H^*_T(X;\QQ)
\end{equation}
where the $t$-variables act by 0 on $\ZZ$. From an orbit harmonics point of view, Equation~\eqref{eq:ordinary-equivariant-relationship} may be interpreted as the deformation $\Zpoints(\aalpha_k) \leadsto \gr \, \II(\Zpoints(\aalpha_k))$ sending the locus $\Zpoints(\aalpha_k)$ linearly to the origin. 

As a concrete example of this approach (and a simplified version of Theorem~\ref{thm:borel-presentation}), let $X$ be the moduli space $X_{n,k}$ of $n$-tuples $\ell_\bullet = (\ell_1, \dots, \ell_n)$ of lines in $\CC^k$ with $\ell_1 + \cdots + \ell_k = \CC^k$. Let  $\aalpha_k = (\alpha_1, \dots, \alpha_k)$ be $k$ distinct rational numbers. Write $\Zpoints_{n,k}(\aalpha_k) \subseteq \QQ^n$ for the locus of points 
\begin{equation}
    \Zpoints_{n,k}(\aalpha_k) := \{ (z_1, \dots, z_n) \in \QQ^n \,:\, \{ z_1, \dots, z_n \} = \{\alpha_1, \dots, \alpha_k \} \}
\end{equation}
whose coordinates consist exactly of the parameters $\alpha_1, \dots, \alpha_k$.
By results of Haglund-Rhoades-Shimozono \cite{HRS} and Pawlowski-Rhoades \cite{PR}, the ordinary cohomology of $X_{n,k}$  is given by
\begin{equation}
    H^*(X_{n,k};\QQ) = \QQ[\xxx_n]/\gr \, \II(\Zpoints_{n,k}(\aalpha_k))
\end{equation}
and the homogeneous ideal $\gr \, \II(\Zpoints_{n,k}(\aalpha_k))
\subseteq \QQ[\xxx_n]$ has generating set
\begin{equation}
    \gr \, \II(\Zpoints_{n,k}(\aalpha_k)) = (x_1^k, \dots, x_n^k, e_n(\xxx_n), e_{n-1}(\xxx_n), \dots, e_{n-k+1}(\xxx_n)).
\end{equation}
Replacing the parameters $\alpha_i$ with variables $t_i$ , we may consider the infinite locus
\begin{equation}
    \Zpoints_{n,k} = \{ (z_1, \dots, z_n, t_1, \dots, t_k) \in \QQ^{n+k} \,:\, \{ z_1, \dots, z_n \} = \{t_1, \dots, t_k \} \}.
\end{equation}
Substituting distinct field elements $t_i \to \alpha_i$ for the $t_i$ recovers the locus $\Zpoints_{n,k}(\aalpha_k)$; taking the linear limit as $t_i \to 0$ recovers the orbit harmonics quotient.
It follows from results in in Section~\ref{sec:Quotient} that the ideal $\II(\Zpoints_{n,k}) \subseteq \QQ[\xxx_n,\ttt_k]$ is generated by 
\begin{itemize}
    \item $(x_i - t_1) \cdots (x_i - t_k)$ for $1 \leq i \leq n$, and
    \item $e_r(\xxx_n) - e_{r-1}(\xxx_n) h_1(\ttt_k) + \cdots + (-1)^r h_r(\ttt_k)$ for $r > n-k$.
\end{itemize}
Theorem~\ref{thm:borel-presentation} implies  that the rational equivariant cohomology of $X_{n,k}$  has presentation
\begin{equation}
    H^*_T(X_{n,k};\QQ) = \QQ[\xxx_n,\ttt_k]/\II(\Zpoints_{n,k}).
\end{equation}

The rest of the paper is organized as follows. In {\bf Section~\ref{sec:Background}} we review background material on symmetric polynomials, commutative algebra, and geometry. In {\bf Section~\ref{sec:Quotient}} we use orbit harmonics to analyze the algebraic quotient ring in Theorem~\ref{thm:borel-presentation}. In {\bf Section~\ref{sec:Cohomology}} we combine these orbit harmonics results with geometric reasoning to prove Theorem~\ref{thm:borel-presentation}.  {\bf Section~\ref{sec:Conclusion}} gives some future directions and open problems.

\section{Background}
\label{sec:Background}

\subsection{Combinatorics} For $n \geq 0$, we abbreviate $[n] := \{1,\dots,n\}$.  
Let $\Stir(n,d)$ be the {\em  Stirling number of the second kind} counting $d$-block set partitions of $[n]$. The number of surjections $[n] \twoheadrightarrow [d]$ is counted by $d! \cdot \Stir(n,d)$.  More generally, for $n \geq k \geq d$ we have
\begin{equation}
    \frac{k!}{(k-d)!} \cdot \Stir(n,d) = \text{number of functions $[n] \to [k]$ whose image has size $d$.}
\end{equation}

A {\em partition} of $n$ is a weakly decreasing sequence $\lambda = (\lambda_1 \geq \cdots \geq \lambda_r)$ of positive integers with $\lambda_1 + \cdots + \lambda_r = n$. We write $\lambda \vdash n$ to indicate that $\lambda$ is a partition of $n$ and write $|\lambda| = n$ for the sum of the parts of $\lambda$. 

Let $\xxx_n = (x_1, \dots, x_n)$ be a list of $n$ variables and let $\ZZ[\xxx_n] = \ZZ[x_1, \dots, x_n]$ be the polynomial ring in these variables with integer coefficients. The symmetric group $\symm_n$ acts on $\ZZ[\xxx_n]$ by subscript permutation. Elements of the invariant ring $\ZZ[\xxx_n]^{\symm_n}$ are called {\em symmetric polynomials}. Given $d > 0$, the degree $d$ {\em elementary} and {\em complete homogeneous} symmetric polynomials are given by
\begin{equation}
    e_d(\xxx_n) := \sum_{1 \leq i_1 < \cdots < i_d \leq n} x_{i_1} \cdots x_{i_d} \quad \text{and} \quad
    h_d(\xxx_n) := \sum_{1 \leq i_1 \leq \cdots \leq i_d \leq n} x_{i_1} \cdots x_{i_d}.
\end{equation}
Either of the sets $\{e_1(\xxx_n), \dots, e_n(\xxx_n)\}$ or $\{h_1(\xxx_n), \dots, h_n(\xxx_n)\}$ freely generates $\ZZ[\xxx_n]^{\symm_n}$ as a $\ZZ$-algebra.
If $\lambda = (\lambda_1 \geq \cdots \geq \lambda_r)$ is a partition of $n$, the {\em Schur polynomial} $s_\lambda(\xxx_n)$ is defined by 
\begin{equation}
    s_\lambda(\xxx_n) := \det \begin{pmatrix} h_{\lambda_i - i + j}(\xxx_n) \end{pmatrix}_{1 \leq i, j \leq r}.
\end{equation}
Here $h_0(\xxx_n) = 1$ and $h_d(\xxx_n) =  0$ for $d < 0$.

\subsection{Commutative algebra} Endow the polynomial ring $\QQ[\xxx_n] = \bigoplus_{d \geq 0} \QQ[\xxx_n]_d$ with its usual degree grading.  For a nonzero polynomial $f \in \QQ[\xxx_n] - \{0\}$ write $\tau(f) \in \QQ[\xxx_n]$ for the top degree homogeneous component of $f$. Explicitly, if $f = f_d + \cdots + f_1 + f_0$ where $f_i$ is homogeneous of degree $i$ and $f_d \neq 0$, we have $\tau(f) = f_d$. If $I \subseteq \QQ[\xxx_n]$ is an ideal, the {\em associated graded ideal} $\gr \, I \subseteq \QQ[\xxx_n]$ has generating set 
\begin{equation}
    \gr \, I := ( \tau(f) \,:\, f \in I, \, f \neq 0).
\end{equation}
The ideal $\gr \, I$ is homogeneous and the quotient $\QQ[\xxx_n]/\gr \, I$ is a graded $\QQ$-algebra. We have the following standard result relating vector space bases of $\QQ[\xxx_n]/I$ and $\QQ[\xxx_n]/\gr \, I$; for a proof see e.g. \cite[Lem. 3.15]{RhoadesViennot}.

\begin{lemma}
    \label{lem:graded-basis} 
    Let $I \subseteq \QQ[\xxx_n]$ be an ideal and let $\BBB$ be a family of homogeneous elements of $\QQ[\xxx_n]$ which descends to a $\QQ$-basis of $\QQ[\xxx_n]/\gr \, I$. Then $\BBB$ descends to a $\QQ$-basis of $\QQ[\xxx_n]/I$.
\end{lemma}

If $\Zpoints \subseteq \QQ^n$ is any subset of affine $n$-space over $\QQ$, we have the vanishing ideal $\II(\Zpoints) \subseteq \QQ[\xxx_n]$ given by
\begin{equation}
    \II(\Zpoints) := \{ f \in \QQ[\xxx_n] \,:\, f(\zzz) = 0 \text{ for all $\zzz \in \Zpoints$ } \}.
\end{equation}
The {\em coordinate ring} of polynomial functions $\Zpoints \to \QQ$ is given by
\begin{equation}
    \QQ[\Zpoints] := \QQ[\xxx_n]/\II(\Zpoints).
\end{equation}
If $\Zpoints$ is finite, by multivariate Lagrange interpolation any function $\Zpoints \to \QQ$ is a polynomial function and we may regard $\QQ[\Zpoints]$ as the $\QQ$-vector space with basis $\Zpoints$. As in the introduction, for $\Zpoints$ finite, orbit harmonics gives a vector space isomorphism
\begin{equation}
    \QQ[\Zpoints] = \QQ[\xxx_n]/\II(\Zpoints) \cong \QQ[\xxx_n]/\gr \, \II(\Zpoints)
\end{equation}
where $\QQ[\xxx_n]/\gr \, \II(\Zpoints)$ is a graded vector space. If $\Zpoints$ is stable under the action of a finite subgroup $G$ of $GL_n(\QQ)$, this is also an isomorphism of $G$-modules.

Let $R$ be any commutative ring and let $M$ be an $R$-module. The module $M$ is {\em free} if it admits an $R$-basis. The cardinality of any such basis is uniquely determined by $M$, and called the {\em rank} of $M$. We will need the following consequence of the Cayley-Hamilton Theorem; see e.g. \cite[Cor. 4.4]{Eisenbud}.

\begin{lemma}
    \label{lem:surjection-lemma}
    Let $R$ be a commutative ring and let $M$ be a free $R$-module of finite rank. Any surjective $R$-module endomorphism $\varphi: M \to M$ is an isomorphism.
\end{lemma}

We review some basic ideas from Gr\"obner theory. A total order $\prec$ on the monomials in $\xxx_n$ is a {\em term order} if 
\begin{itemize}
    \item we have $1 \preceq m$ for all monomials $m$, and
    \item if $m_1 \preceq m_2$, we have $m_1 m_3 \preceq m_2 m_3$ for all monomials $m_1, m_2, m_3$. 
\end{itemize}
The only term order we will use in this paper is the {\em lexicographical} term order defined by $x_1^{a_1} \cdots x_n^{a_n} \prec x_1^{b_1} \cdots x_n^{b_n}$ if there exists $1 \leq i \leq n$ such that $a_1 = b_1, \dots, a_{i-1} = b_{i-1}$ and $a_i < b_i$.

Let $\prec$ be a term order on the monomials of $\QQ[\xxx_n]$.
If $f \in \QQ[\xxx_n]$ is a nonzero, write $\initial_\prec(f)$ for the $\prec$-largest monomial appearing in $f$. If $I \subseteq \QQ[\xxx_n]$ is an ideal, the {\em initial ideal} is the monomial ideal $\initial_\prec(I) \subseteq \QQ[\xxx_n]$ with generating set
\begin{equation}
    \initial_\prec(I) := ( \initial_\prec(f) \,:\, f \in I, \, f \neq 0 ).
\end{equation}
A finite set $G = \{g_1, \dots, g_r\} \subset I$ of nonzero elements in $I$ is a {\em Gr\"obner basis} of $I$ if 
\begin{equation}
    (\initial_\prec(g_1), \dots, \initial_\prec(g_r)) = \initial_\prec(I)
\end{equation}
as ideals in $\QQ[\xxx_n]$. In this case we have $I = (G)$. The set $\BBB$ of monomials in $\QQ[\xxx_n]$ which do not belong to $\initial_\prec(I)$ descends to a vector space basis of $\QQ[\xxx_n]/I$. The basis $\BBB$ is the {\em standard monomial basis} of $\QQ[\xxx_n]/I$ with respect to the term order $\prec$.

\subsection{Cohomology} Let $X$ be a smooth complex algebraic variety. An {\em affine paving} of $X$ is a filtration
\begin{equation}
    \varnothing = X_0 \subseteq X_1 \subseteq \cdots \subseteq X_m = X
\end{equation}
of $X$ by closed subvarieties such that $X_i - X_{i-1}$ is isomorphic to a disjoint union of affine spaces (possibly of varying dimensions) for each $i$. These affine spaces are the {\em cells} of the affine paving. Since $X$ is smooth, the Zariski closure $\overline{C}$ of any cell $C$ yields a class $[\overline{C}]$ in the ordinary cohomology ring $H^*(X;\ZZ)$ of cohomological degree twice the codimension of $C$ in $X$.

\begin{lemma}
    \label{lem:affine-paving-cohomology}
    Let $X$ be a smooth complex variety which admits an affine paving. The ordinary cohomology ring $H^*(X;\ZZ)$ is a free $\ZZ$-module with basis given by the classes $[\overline{C}]$ where $C$ ranges over the cells of the paving. 
\end{lemma}

Now let $T = (\CC^*)^k$ be the rank $k$ complex torus and assume that $X$ carries an algebraic action of $T$. 
Let $ET$ be a contractible space which carries a free right action of $T$. We define the quotient space $ET \times_T X := ET \times X/\sim$ where $(e,x) \sim (e \cdot t^{-1}, t \cdot x)$ for all $e \in ET$, $x \in X$, and $t \in T$. The {\em equivariant cohomology ring} is defined by
\begin{equation}
    H^*_T(X;\ZZ) := H^*(ET \times_T X;\ZZ)
\end{equation}
where $H^*(ET \times_T X;\ZZ)$ is the usual singular cohomology. For an excellent textbook treatment of equivariant cohomology, we refer the reader to Anderson and Fulton \cite{AF}.

When $X = \{\bullet\}$ is a one-point space, the quotient $ET \times_T \{ \bullet \} = BT = (\PP^{\infty})^k$ may be identified with the classifying space of $T$. Here $\PP^{\infty}$ is the infinite-dimensional complex projective space. Therefore, we have
$H^*_T( \{\bullet\}; \ZZ) = \ZZ[\ttt_k]$
where $\ttt_k = (t_1, \dots, t_k)$. If $X$ is any space with a $T$-action, the map $X \to \{\bullet\}$ endows $H^*_T(X;\ZZ)$ with the structure of a $\ZZ[\ttt_k]$-module.

An affine paving of $X$ is {\em $T$-invariant} if all of its cells are closed under the action of $T$. As before, the Zariski closures of the cells of a $T$-invariant paving give classes in $H_T^*(X;\ZZ)$. Lemma~\ref{lem:affine-paving-cohomology} extends to the equivariant setting. For a proof of the next result, see e.g. \cite[Cpt. 4, Prop. 7.1]{AF}.

\begin{lemma}
    \label{lem:affine-paving-equivariant}
    Let $X$ be a smooth complex variety with an action of $T$ which admits a $T$-invariant affine paving. The equivariant cohomology ring $H^*_T(X;\ZZ)$ is a free $\ZZ[\ttt_k]$-module with basis given by the classes $[\overline{C}]$ where $C$ ranges over the cells of the paving. 
\end{lemma}

Let $X$ be a smooth complex variety and let $\EEE \twoheadrightarrow X$ be a rank $r$ complex vector bundle over $X$. For any point $x \in X$, the fiber $\EEE_x$ of $\EEE$ over $x$ is an $r$-dimensional complex vector space. A vector bundle of rank $1$ is called a {\em line bundle}. For each $i > 0$, we have the {\em Chern class} $c_i(\EEE) \in H^{2i}(X;\ZZ)$. We have $c_i(\EEE) = 0$ whenever $i > r$. The {\em total Chern class} is the element $c(\EEE) \in H^*(X;\ZZ)$ given by 
$c(\EEE) := \sum_{i \geq 0} c_i(\EEE)$ where $c_0(\EEE) := 1$. 
The precise construction of Chern classes may be found in  \cite{AF}; for our purposes we will need the following facts.
\begin{enumerate}
    \item If $\EEE = \CC^r$ is the trivial bundle, then $c(\EEE) = 1$.
    \item  If $0 \to \EEE_1 \to \EEE_2 \to \EEE_3 \to 0$ is a short exact sequence of vector bundles, we have $$c(\EEE_2) = c(\EEE_1) \cdot c(\EEE_3).$$
\end{enumerate}
Item (2) implies that $c(\EEE \oplus \EEE') = c(\EEE) \cdot c(\EEE')$; this relation is called the {\em Whitney Sum Formula}.

Now suppose $X$ carries an action of $T$ and that $\EEE \twoheadrightarrow X$ is a $T$-equivariant complex vector bundle of rank $r$. Then $ET \times_T \EEE$ is an ordinary rank $r$ vector bundle over $ET \times_T X$. For $i > 0$ we have the {\em equivariant Chern class} $c_i^T(\EEE) \in H^{2i}_T(X;\ZZ)$ defined by $$c_i^T(\EEE) = c_i(ET \times_T \EEE) \in H^{2i}(ET \times_T X;\ZZ) = H^{2i}_T(X;\ZZ).$$ As before, we have $c_i^T(\EEE) = 0$ for $i > r$ and define $c^T(\EEE) := \sum_{i \geq 0} c_i^T(\EEE)$ where $c_0^T(\EEE) := 1$. Equivariant versions of the facts from the last paragraph read as follows.
\begin{enumerate}
    \item If $\EEE = \CC^k$ is the trivial bundle $\CC^k \times X$ of rank $k$ equipped with the natural action of $T = (\CC^*)^k$, we have $c^T(\EEE) = (1 + t_1) \cdots (1 + t_k)$.
    \item If $0 \to \EEE_1 \to \EEE_2 \to \EEE_3 \to 0$ is a short exact sequence of equivariant bundles, we have $$c^T(\EEE_2) = c^T(\EEE_1) \cdot c^T(\EEE_3).$$
\end{enumerate}
In particular, we have $c^T(\EEE \oplus \EEE') = c^T(\EEE) \cdot c^T(\EEE')$.

Let $\EEE \twoheadrightarrow X$ be an equivariant vector bundle. If $\EEE$ is an equivariant direct sum of line bundles $\EEE = \LLL_1 \oplus \cdots \oplus \LLL_r$, we have $c^T(\EEE) = (1 + x_1) \cdots (1 + x_r)$ where $x_i \in H^2_T(X;\ZZ)$ is given by $x_i = c^T_1(\LLL_i)$. Although not every vector bundle splits into a direct sum of line bundles, by choosing an appropriate filtration on $\EEE$ we can form a flag extension $X'$ of $X$ such that
\begin{itemize}
    \item we have a $T$-equivariant surjection $X' \twoheadrightarrow X$ and the induced map $H^*_T(X;\ZZ) \hookrightarrow H^*_T(X';\ZZ)$ is an embedding, and 
    \item there exist elements $x_1, \dots, x_r \in H^2_T(X';\ZZ)$ such that $c^T(\EEE) = (1 + x_1) \cdots (1 + x_r)$.
\end{itemize}
The elements $x_1, \dots, x_r \in H^2_T(X';\ZZ)$ are the {\em equivariant Chern roots} of $\EEE$. For any $i > 0$ we have $e_i(x_1, \dots, x_r) = c_i^T(\EEE) \in H^*_T(X;\ZZ)$, so any symmetric polynomial in $x_1, \dots, x_r$ lies in $H^*_T(X;\ZZ)$.

Our moduli space of study $X_{n,k,d}$ is the total space of a fiber bundle $X_{n,k,d} \to \Gr(d,\CC^k)$; see Observation~\ref{obs:fiber-bundle}. To understand the cohomology of $X_{n,k,d}$, we will use the following standard result on bundle topology; see e.g. \cite[Thm. 4D.1]{Hatcher}.

\begin{theorem}
    \label{thm:leray-hirsch}
    {\em (Leray-Hirsch)} Let $F \to E \xrightarrow{ \, \pi \,} B$ be a fiber bundle and let $R$ be a coefficient ring. Assume that
    \begin{itemize}
        \item for each $d$, the cohomology group $H^d(F;R)$ is a free $R$-module, and
        \item there exist classes $c_j \in H^{d}(E;R)$ whose restrictions $\iota^*(c_j) \in H^{d}(F;R)$ form an $R$-basis in each fiber $F$, where $\iota: F \hookrightarrow E$ is the inclusion.
    \end{itemize}
    The map $\Phi: H^*(B;R) \otimes_R H^*(F;R) \to H^*(E;R)$ given by $\sum_{i,j} b_i \otimes \iota^*(c_j) \mapsto \pi^*(b_i) \cdot c_j$ is an isomorphism of $H^*(B;R)$-modules.
\end{theorem}

\section{Quotient Rings and Orbit Harmonics}
\label{sec:Quotient}

\subsection{A finite locus} The starting point for presenting the equivariant cohomology of $X_{n,k,d}$ is calculating the orbit harmonics quotient ring of a strategically chosen finite point locus. This locus carries an action of $\symm_d$, and taking $\symm_d$-invariants will yield the ordinary cohomology of $X_{n,k,d}$.

\begin{defn}
    \label{def:z-locus-specialized}
    Let $\aalpha_k = (\alpha_1, \dots, \alpha_k) \in \QQ^k$ be a list of $k$ distinct rational numbers. Define $\Zpoints_{n,k,d}(\aalpha_k) \subseteq \QQ^{n+d}$ to be the locus of points $(z_1, \dots, z_n; z_{n+1}, \dots, z_{n+d})$ such that
    \begin{itemize}
        \item every coordinate $z_i$ lies in $\{\alpha_1, \dots, \alpha_k \}$,
        \item the last coordinates $z_{n+1}, \dots, z_{n+d}$ are distinct, and
        \item we have $\{z_1, \dots, z_n\} = \{z_{n+1}, \dots, z_{n+d} \}.$
    \end{itemize}
\end{defn}

For example, if $n = 3, k = 3,$ and $d = 2$, the locus $\Zpoints_{n,k,d}(\aalpha_k) \subseteq \QQ^{3+2}$ consists of the 36 points
$$\left\{ \begin{tiny} \begin{array}{c} 
(\alpha_1, \alpha_1, \alpha_2; \alpha_1, \alpha_2), \quad (\alpha_1, \alpha_2, \alpha_1; \alpha_1, \alpha_2), \quad 
(\alpha_2, \alpha_1, \alpha_1; \alpha_1, \alpha_2), \quad (\alpha_1, \alpha_2, \alpha_2; \alpha_1, \alpha_2), \quad (\alpha_2, \alpha_1, \alpha_2; \alpha_1, \alpha_2), \quad (\alpha_2, \alpha_2, \alpha_1; \alpha_1, \alpha_2), \\
(\alpha_1, \alpha_1, \alpha_2; \alpha_2, \alpha_1), \quad (\alpha_1, \alpha_2, \alpha_1; \alpha_2, \alpha_1), \quad 
(\alpha_2, \alpha_1, \alpha_1; \alpha_2, \alpha_1), \quad (\alpha_1, \alpha_2, \alpha_2; \alpha_2, \alpha_1), \quad (\alpha_2, \alpha_1, \alpha_2; \alpha_2, \alpha_1), \quad (\alpha_2, \alpha_2, \alpha_1; \alpha_2, \alpha_1), \\
(\alpha_1, \alpha_1, \alpha_3; \alpha_1, \alpha_3), \quad (\alpha_1, \alpha_3, \alpha_1; \alpha_1, \alpha_3), \quad 
(\alpha_3, \alpha_1, \alpha_1; \alpha_1, \alpha_3), \quad (\alpha_1, \alpha_3, \alpha_3; \alpha_1, \alpha_3), \quad (\alpha_3, \alpha_1, \alpha_3; \alpha_1, \alpha_3), \quad (\alpha_3, \alpha_3, \alpha_1; \alpha_1, \alpha_3), \\
(\alpha_1, \alpha_1, \alpha_3; \alpha_3, \alpha_1), \quad (\alpha_1, \alpha_3, \alpha_1; \alpha_3, \alpha_1), \quad 
(\alpha_3, \alpha_1, \alpha_1; \alpha_3, \alpha_1), \quad (\alpha_1, \alpha_3, \alpha_3; \alpha_3, \alpha_1), \quad (\alpha_3, \alpha_1, \alpha_3; \alpha_3, \alpha_1), \quad (\alpha_3, \alpha_3, \alpha_1; \alpha_3, \alpha_1) \\
(\alpha_2, \alpha_2, \alpha_3; \alpha_2, \alpha_3), \quad (\alpha_2, \alpha_3, \alpha_2; \alpha_2, \alpha_3), \quad 
(\alpha_3, \alpha_2, \alpha_2; \alpha_2, \alpha_3), \quad (\alpha_2, \alpha_3, \alpha_3; \alpha_2, \alpha_3), \quad (\alpha_3, \alpha_2, \alpha_3; \alpha_2, \alpha_3), \quad (\alpha_3, \alpha_3, \alpha_2; \alpha_2, \alpha_3), \\
(\alpha_2, \alpha_2, \alpha_3; \alpha_3, \alpha_2), \quad (\alpha_2, \alpha_3, \alpha_2; \alpha_3, \alpha_2), \quad 
(\alpha_3, \alpha_2, \alpha_2; \alpha_3, \alpha_2), \quad (\alpha_2, \alpha_3, \alpha_3; \alpha_3, \alpha_2), \quad (\alpha_3, \alpha_2, \alpha_3; \alpha_3, \alpha_2), \quad (\alpha_3, \alpha_3, \alpha_2; \alpha_3, \alpha_2)
\end{array} \end{tiny}  \right\}$$
The points in $\Zpoints_{n,k,d}(\aalpha_k)$ may be enumerated as follows. We have $\frac{k!}{(k-d)!}$ choices for the last $d$ coordinates of a point $(z_1, \dots, z_n; z_{n+1}, \dots, z_{n+d}) \in \Zpoints_{n,k,d}(\aalpha_k)$. Once these terminal coordinates are fixed, the first $n$ coordinates correspond to a surjective map $[n] \to [d]$; the number of such maps is $d! \cdot \Stir(n,d)$. We therefore have
\begin{equation}
\label{eq:z-count}
    |\Zpoints_{n,k,d}(\aalpha_k)| = \frac{k! \cdot d!}{(k-d)!} \cdot \Stir(n,d).
\end{equation}
The locus $\Zpoints_{n,k,d}(\aalpha_k)$ carries an action of $\symm_d$ on its last $d$ coordinates. This action is free because the coordinates $z_{n+1}, \dots, z_{n+d}$ are distinct, and the number of $\symm_d$-orbits is given by
\begin{equation}
\label{eq:z-orbit-count}
    |\Zpoints_{n,k,d}(\aalpha_k)/\symm_d| = \frac{1}{d!} \cdot |\Zpoints_{n,k,d}(\aalpha_k)| = \frac{k!}{(k-d)!} \cdot \Stir(n,d).
\end{equation}

Let $\xxx_n = (x_1, \dots, x_n)$, $\yyy_d = (y_1, \dots, y_d)$ and identify the coordinate ring of $\QQ^{n+d}$ with $\QQ[\xxx_n,\yyy_d]$.  The vanishing ideal $\II(\Zpoints_{n,k,d}(\aalpha_k))$ is a subset of $\QQ[\xxx_n,\yyy_d]$. For orbit harmonics purposes, we want to calculate the associated graded ideal $\gr \, \II(\Zpoints_{n,k,d}(\aalpha_k))$. We begin by giving generators for the inhomogeneous ideal $\II(\Zpoints_{n,k,d}(\aalpha_k))$.

\begin{lemma}
    \label{lem:orbit-harmonics-Z-inhomogeneous}
    The ideal $\II(\ZZZ_{n,k,d} (\aalpha_k)) \subseteq \QQ[\xxx_n,\yyy_d]$ is generated by
    \begin{itemize}
    \item $e_r(\aalpha_k) - e_{r-1}(\aalpha_k) h_1(\yyy_d) + \cdots + (-1)^r h_r(\yyy_d)$ for all $r > k-d$, 
    \item $e_r(\xxx_n) - e_{r-1}(\xxx_n) h_1(\yyy_d) + \cdots + (-1)^r h_r(\yyy_d) = 0$ for all $r > n-d$, and
    \item $x_i^d - x_i^{d-1} e_1(\yyy_d) + \cdots + (-1)^d e_d(\yyy_d)$ for $i = 1, \dots, n$.
    \end{itemize}
\end{lemma}

\begin{proof}
    Let $u$ be a new variable. The generators in the first bullet point vanish on $(x_1, \dots, x_n, y_1, \dots, y_d)$ if and only if the rational function
    \[ \frac{(1 + \alpha_1 u) \cdots (1 + \alpha_k u)}{(1 + y_1 u) \cdots (1 + y_d u)}\]
    is a polynomial in $u$ of degree $\leq k-d$. This happens if and only if $y_1, \dots, y_d$ are $d$ distinct elements of $\{ \alpha_1, \dots, \alpha_k \}$. Similarly, the relations in the second bullet point vanish if and only if the rational function
        \[ \frac{(1 + x_1 u) \cdots (1 + x_n u)}{(1 + y_1 u) \cdots (1 + y_d u)}\]
    is a polynomial in $u$ of degree $\leq n-d$, i.e. the coordinates $y_1, \dots, y_d$ each occur among the coordinates $x_1, \dots, x_n$. The generators in the first two bullet points already suffice to cut out $\ZZZ_{n,k,d}(\aalpha_k)$. The elements in the last bullet point are also members of $\II(\ZZZ_{n,k,d}(\aalpha_k))$ because
    \[ \frac{(1 + y_1 u) \cdots (1 + y_d u)}{(1 + x_i u)}\]
    is a polynomial in $u$ of degree $d-1$ on $\ZZZ_{n,k,d}(\aalpha_k)$.
\end{proof}

While the generators in the third bullet point of Lemma~\ref{lem:orbit-harmonics-Z-inhomogeneous} were redundant, they will be helpful in obtaining the associated graded ideal of $\II(\Zpoints_{n,k,d}(\aalpha_k))$. We introduce notation for the ideal generated by the highest degree components of the polynomials in Lemma~\ref{lem:orbit-harmonics-Z-inhomogeneous} as follows.

\begin{defn}
    \label{def:j-ideal-definition}
    Let $J_{n,k,d}^\QQ \subseteq \QQ[\xxx_n,\yyy_d]$ be the ideal generated by 
    \begin{itemize}
    \item $h_r(\yyy_d)$ for all $r > k-d$, 
    \item $e_r(\xxx_n) - e_{r-1}(\xxx_n) h_1(\yyy_d) + \cdots + (-1)^r h_r(\yyy_d) = 0$ for all $r > n-d$, and
    \item $x_i^d - x_i^{d-1} e_1(\yyy_d) + \cdots + (-1)^d e_d(\yyy_d)$ for $i = 1, \dots, n$.
    \end{itemize}
\end{defn}

The superscript in $J_{n,k,d}^\QQ$ is meant to recall the ground field of the ambient ring. Observe that the generators of $J_{n,k,d}^\QQ$ do not depend on the parameters $\aalpha_k = (\alpha_1, \dots, \alpha_k)$. Lemma~\ref{lem:orbit-harmonics-Z-inhomogeneous} implies that $J_{n,k,d}^\QQ \subseteq \gr \, \II(\Zpoints_{n,k,d}(\aalpha_k))$ for any list $\aalpha_k \in \QQ^k$ of distinct rational parameters; our next aim is to show that this containment of ideals is an equality. To this end, we introduce some combinatorics and recall some algebraic results of Haglund-Rhoades-Shimozono \cite{HRS}. 

Recall that a {\em shuffle} of two sequences $(a_1, \dots, a_r)$ and $(b_1, \dots, b_s)$ is an interleaving $(c_1, \dots, c_{r+s})$ of these sequences which preserves the relative order of the $a$'s and the $b$'s. For integers $d \leq n$, an {\em $(n,d)$-staircase} is a shuffle $(s_1, \dots, s_n)$ of the sequences $(0,1,\dots,d-1)$ and $(d-1, \dots, d-1)$ where the latter sequence has $n-d$ copies of $d-1$. For example, the $(5,3)$-staircases consist of the following shuffles of $(0,1,2)$ and $(2,2)$:
$$(0,1,2,2,2), \quad (0,2,1,2,2), \quad (2,0,1,2,2), \quad (0,2,2,1,2), \quad (2,0,2,1,2), \quad (2,2,0,1,2).$$

A length $n$ sequence $(e_1, \dots, e_n)$ of nonnegative integers is {\em $(n,d)$-substaircase} if it is componentwise $\leq$ at least one $(n,d)$-staircase. A monomial $x_1^{e_1} \cdots x_n^{e_n}$ in the variables $\xxx_n = (x_1, \dots, x_n)$ is {\em $(n,d)$-substaircase} if its exponent sequence $(e_1,\dots,e_n)$ has the corresponding property. For example, when $(n,d) = (3,2)$ we have the staircases $(0,1,1), (1,0,1)$ and the substaircase monomials
$$\{x_2 x_3, x_1 x_3, x_1 ,x_2, x_3, 1\}.$$
The number of $(n,d)$-stubstaircase monomials coincides with the number of surjections $[n] \to [d]$.

\begin{lemma}
\label{lem:substaircase-count} \cite[Lem. 4.8]{HRS}
The number of $(n,d)$-substaircase monomials is $d! \cdot \Stir(n,d)$.
\end{lemma}

In fact, it is shown in \cite{HRS} that the $(n,d)$-substaircase monomials form the standard monomial basis of an orbit harmonics quotient ring $R_{n,d}$ coming from a locus in bijective correspondence with ordered set partitions of $\{1,\dots,n\}$ with $d$ blocks. The generating set of the defining ideal of $R_{n,d}$ is as in the following result.

\begin{lemma}
    \label{lem:substaircase-standard-monomial} \cite[Lem. 3.5, Lem. 4.8, Thm. 4.14]{HRS}
    Fix positive integers $d \leq n$ and endow the monomials in $\QQ[\xxx_n]$ with the lexicographical term order. The family of $(n,d)$-substaircase monomials is the standard monomial basis for the quotient ring
    $$\QQ[\xxx_n] / (x_1^d, \dots, x_n^d, e_n(\xxx_n), e_{n-1}(\xxx_n), \dots, e_{n-d+1}(\xxx_n)).$$
\end{lemma}

Finally, for the purposes of presenting integral (rather than merely rational) cohomology, it will be useful to have the following integrality statement related to Lemma~\ref{lem:substaircase-standard-monomial}. 

\begin{lemma}
    \label{lem:integral-substaircase} \cite[Lem. 3.4, Thm. 4.14]{HRS}
    Fix positive integers $d \leq n$. The ideal $$(x_1^d, \dots, x_n^d, e_n(\xxx_n), e_{n-1}(\xxx_n), \dots, e_{n-d+1}(\xxx_n)) \subseteq \QQ[\xxx_n]$$ appearing in Lemma~\ref{lem:substaircase-standard-monomial} has a lexicographical Gr\"obner basis $G$ consisting of the variable powers $x_1^d, \dots, x_n^d$ together with a family of polynomials of the form 
    $$\sum_{r \, = \, n-d+1}^n g_r(\xxx_n) e_r(\xxx_n)$$
    where $g_r(\xxx_n) \in \ZZ[\xxx_n]$ are homogeneous and have integer coefficients.
\end{lemma}

The precise form of the polynomials $g_r(\xxx_n)$ and the Gr\"obner basis elements $\sum_{r \, = \, n-d+1}^n g_r(\xxx_n) e_r(\xxx_n)$ appearing in Lemma~\ref{lem:integral-substaircase} are not important for us. In fact, the sums $\sum_{r \, = \, n-d+1}^n g_r(\xxx_n) e_r(\xxx_n)$ are certain {\em Demazure characters} (or {\em key polynomials}) and the coefficients $g_r(\xxx_n)$ are also Demazure characters (up to sign). The crucial fact for our purposes is that the polynomials $g_r(\xxx_n)$ and $\sum_{r \, = \, n-d+1}^n g_r(\xxx_n) e_r(\xxx_n)$ have integer coefficients. A more basic integrality result is as follows.

\begin{lemma}
    \label{lem:h-integrality}
    Let $\yyy = (y_1, \dots, y_d)$ and let $a \geq 0$. Consider the ideal $I \subseteq \ZZ[\yyy_d]$ generated by $h_d(\yyy_d)$ for all $d > a$. For all $1 \leq i \leq n$ we have $h_{a+i}(y_i, y_{i+1}, \dots, y_d) \in I$.
\end{lemma}

We include the standard proof of Lemma~\ref{lem:h-integrality}.

\begin{proof}
    In fact, we claim that $h_{d+i}(y_i, y_{i+1}, \dots, y_d) \in I$ for all $1 \leq i \leq n$ and all $d > a$. If $i = 1$ this is true by the definition of $I$. If $d > a$ and $i > 1$ we have
    \begin{equation}
       h_{d+i}(y_i, \dots, y_n) = h_{d+i}(y_{i-1}, y_i, \dots, y_n) - y_{i-1} h_{d+i-1}(y_{i-1},y_i, \dots, y_n).
    \end{equation}
    By induction on $i$ and since $d > a$, we have $h_{d+i}(y_{i-1}, y_i, \dots, y_n), h_{d+i-1}(y_{i-1},y_i, \dots, y_n) \in I$ and the lemma follows.
\end{proof}

We are ready to give generators for the orbit harmonics ideal of the locus $\Zpoints_{n,k,d}(\aalpha_k)$ and describe the standard monomial basis of the relevant quotient ring. This will turn out to be the following set of monomials.

\begin{defn}
    \label{defn:a-monomials}
    For positive integers $n \geq k \geq d$, let $\AAA_{n,k,d}$ be the following set of monomials in $\QQ[\xxx_n,\yyy_d]$:
    \begin{equation}
    \AAA_{n,k,d} := \{ x_1^{a_1} \cdots x_n^{a_n} \cdot y_1^{b_1} \cdots y_d^{b_d} \,:\, (a_1, \dots, a_n) \text{ is $(n,d)$-substaircase and } b_i < k-d+i \}.
    \end{equation}
\end{defn}

The monomials in $\AAA_{n,k,d}$ may be enumerated as follows. The number of choices for the exponent sequence $(a_1,\dots,a_n)$ of the $x$-variables equals the number of $(n,d)$-substaircase sequences. We have $\frac{k!}{(k-d)!}$ choices for the exponent sequence $(b_1,\dots,b_d)$ of the $y$-variables. By Lemma~\ref{lem:substaircase-count} we see that
\begin{multline} \label{eq:a-count}
|\AAA_{n,k,d}| = \text{(number of $(n,d)$-substaircase sequences)} \times \frac{k!}{(k-d)!} \\ = \frac{k! \cdot d!}{(k-d)!} \cdot |\Stir(n,d)| = |\Zpoints_{n,k,d}(\aalpha_k)|\end{multline}
where the last equality used Equation~\eqref{eq:z-count}.

\begin{proposition}
    \label{prop:orbit-harmonics-Z-homogeneous}
    Let $\aalpha_k \in \QQ^k$ be a list of $k$ distinct rational numbers. We have the equality of ideals $\gr \, \II(\Zpoints_{n,k,d}(\aalpha_k)) = J_{n,k,d}^\QQ$ in $\QQ[\xxx_n,\yyy_d]$ and an isomorphism
    \begin{equation}
        \QQ[\Zpoints_{n,k,d}(\aalpha_k)] = \QQ[\xxx_n,\yyy_d]/\II(\Zpoints_{n,k,d}(\aalpha_k)) \cong \QQ[\xxx_n,\yyy_d]/\gr \, \II(\Zpoints_{n,k,d}(\aalpha_k))
    \end{equation}
    of $\symm_d$-modules. The set $\AAA_{n,k,d}$ is the standard monomial basis of $\QQ[\xxx_n,\yyy_d]/\gr \, \II(\Zpoints_{n,k,d}(\aalpha_k))$ with respect to the lexicographical order $\prec$ with the variable order $x_1 > \cdots > x_n > y_1 > \cdots > y_d$.
\end{proposition}

\begin{proof}
    The module isomorphism follows from the ideal equality.
    Lemma~\ref{lem:orbit-harmonics-Z-inhomogeneous} gives the containment of ideals $J_{n,k,d}^\QQ \subseteq \gr \, \II(\ZZZ_{n,k,d}(\aalpha_k))$. We will show that 
    \begin{multline} 
    \dim_\QQ \QQ[\xxx_n,\yyy_d]/J_{n,k,d}^\QQ \leq \dim_\QQ \QQ[\xxx_n,\yyy_d]/ \gr \, \II(\Zpoints_{n,k,d}(\aalpha_k)) 
 \\ = |\ZZZ_{n,k,d}(\aalpha_k)| =  \frac{k! \cdot d!}{(k-d)!} \cdot \Stir(n,d)
    \end{multline} 
    by establishing that the set $\AAA_{n,k,d}$ contains the $\prec$-standard monomial basis of $\QQ[\xxx_n,\yyy_d]/J_{n,k,d}^\QQ$.
    
    By Lemma~\ref{lem:h-integrality}, the polynomial $h_{k-d+i}(y_i, y_{i+1}, \dots, y_d)$ lies in the ideal generated by $h_r(\yyy_d)$ for $r > k-d$, so that in particular
    $$h_{k-d+i}(y_i, y_{i+1}, \dots, y_d) \in J_{n,k,d}^\QQ \quad \text{for $1 \leq i \leq d$.}$$
    Since
    $\initial_\prec h_{k-d+i}(y_i, y_{i+1}, \dots, y_d) = y_i^{k-d+i}$ we have $y_i^{k-d+i} \in \initial_\prec(J_{n,k,d}^\QQ)$ for $1 \leq i \leq d$.
    For any element  
    \begin{equation}x_i^d \quad \text{or} \quad \sum_{r \, = \, n-d+1}^n g_r(\xxx_n) e_r(\xxx_n)\end{equation}
    of the Gr\"obner basis in Lemma~\ref{lem:integral-substaircase}, we have a corresponding element 
    \begin{equation}
        x_i^d - x_i^{d-1} e_1(\yyy_d) + \cdots + (-1)^d e_d(\yyy_d) \quad \text{or} \quad 
    \sum_{r \, = \, n-d+1}^n g_r(\xxx_n) \cdot \left( \sum_{a+b \, = \, r} (-1)^b e_a(\xxx_n) h_b(\yyy_d) \right)
    \end{equation}
    of $J_{n,k,d}^\QQ$ which has the same $\prec$-leading monomial (here we use the fact that $\prec$ weights $x$-variables more heavily than $y$-variables).
    By Lemmas~\ref{lem:substaircase-standard-monomial} and \ref{lem:integral-substaircase}, the standard monomial basis of $\QQ[\xxx_n,\yyy_d]/J_{n,k,d}^\QQ$ is a subset $\AAA_{n,k,d}$.
    Equation~\eqref{eq:a-count} implies that
    \begin{multline}
        \dim_\QQ \QQ[\xxx_n,\yyy_d]/J_{n,k,d}^\QQ \leq |\AAA_{n,k,d}| = |\ZZZ_{n,k,d}(\aalpha_{n,k,d})| \\ = \dim_\QQ \QQ[\xxx_n,\yyy_d]/\gr \, \II(\Zpoints_{n,k,d}(\aalpha_k)) \leq \dim_\QQ \QQ[\xxx_n,\yyy_d]/J_{n,k,d}^\QQ.
    \end{multline}
    It follows that $\AAA_{n,k,d}$ is precisely the standard monomial basis of $\QQ[\xxx_n,\yyy_d]/J_{n,k,d}^\QQ$ and that we have the required  equality $\gr \, \II(\Zpoints_{n,k,d}(\aalpha_k)) = J_{n,k,d}^\QQ$ of ideals in $\QQ[\xxx_n,\yyy_d]$. 
\end{proof}

In geometric terms, the $y$-variables will ultimately correspond to Chern roots of the tautological bundle $\UUU_d$ over $\Gr(d,\CC^k)$. To work with Chern classes, we  need to take $\symm_d$-invariants. Let $(-)^\natural$ be the {\em Reynolds operator}
\begin{equation}
    f^\natural := \frac{1}{d!} \sum_{w \in \symm_d} w \cdot f
\end{equation}
on any $\symm_d$-module defined over $\QQ$. The following basic lemma relates ideals in algebras with an $\symm_d$-action and ideals in their invariant subalgebras. 

\begin{lemma}
    \label{lem:rational-invariant-generation}
    Let $R$ be a $\QQ$-algebra which carries an action of $\symm_d$ and let $R^{\symm_d} \subseteq R$ be the $\symm_d$-invariant subalgebra. Let $G \subseteq R^{\symm_d}$ be a set of $\symm_d$-invariant elements. We have the equality
    $$ R \cdot G \cap R^{\symm_d} = R^{\symm_d} \cdot G $$
    of ideals of the invariant ring $R^{\symm_d}$.
\end{lemma}

The standard proof of this lemma is as follows.

\begin{proof}
    Certainly $R^{\symm_d} \cdot G \subseteq R \cdot G \cap R^{\symm_d}$. For the reverse containment, let $f \in R \cdot G \cap R^{\symm_d}$. There exists an expression of the form
    \begin{equation}
        f = \sum_{g \, \in \, G} h_g \cdot g
    \end{equation}
    for some $h_g \in R$. Applying the Reynolds operator and using the fact that both $f$ and the elements of $G$ are $\symm_d$-invariant we get
    \begin{equation}
        f = \sum_{g \, \in \, G} (h_g)^\natural \cdot g
    \end{equation}
    so that $f \in R^{\symm_d} \cdot G$ as required.
\end{proof}

The inhomogeneous ideal $\II(\Zpoints_{n,k,d}(\aalpha_k))$ and the homogeneous ideal $\gr \, \II(\Zpoints_{n,k,d}(\aalpha_k))$ are both stable under the action of $\symm_d$ on $\QQ[\xxx_n,\yyy_d]$. We have ideals 
$\II(\Zpoints_{n,k,d}(\aalpha_k)) \cap \QQ[\xxx_n,\yyy_d]^{\symm_d}$ and $\gr \, \II(\Zpoints_{n,k,d}(\aalpha_k)) \cap \QQ[\xxx_n,\yyy_d]^{\symm_d}$ in the $\symm_d$-invariant subring. 

\begin{lemma}
    \label{lem:invariant-ideal-generators} As an ideal in $\QQ[\xxx_n,\yyy_d]^{\symm_d}$, the intersection $\II(\Zpoints_{n,k,d}(\aalpha_k)) \cap \QQ[\xxx_n,\yyy_d]^{\symm_d}$ has the same generators as in Lemma~\ref{lem:orbit-harmonics-Z-inhomogeneous}. As an ideal in $\QQ[\xxx_n,\yyy_d]^{\symm_d}$, the intersection 
    $$\gr \, \II(\Zpoints_{n,k,d}(\aalpha_k)) \cap \QQ[\xxx_n,\yyy_d]^{\symm_d} = J_{n,k,d}^\QQ \cap \QQ[\xxx_n,\yyy_d]^{\symm_d}$$
    has the same generating set as that of $J_{n,k,d}^\QQ$ in Definition~\ref{def:j-ideal-definition}.
\end{lemma}

\begin{proof}
    The displayed equality of ideals is part of Proposition~\ref{prop:orbit-harmonics-Z-homogeneous}.
    The generators appearing in Lemma~\ref{lem:orbit-harmonics-Z-inhomogeneous} and Definition~\ref{def:j-ideal-definition} are $\symm_d$-invariant. Now apply Lemma~\ref{lem:rational-invariant-generation}.
\end{proof}

We may regard both of the quotient rings 
\begin{equation}
    \label{eq:two-vector-spaces}
    \QQ[\xxx_n,\yyy_d]^{\symm_d}/(J_{n,k,d}^\QQ \cap \QQ[\xxx_n,\yyy_d]^{\symm_d}) \quad \text{and} \quad \QQ[\xxx_n,\yyy_d]^{\symm_d}/(\II(\Zpoints_{n,k,d}(\aalpha_k)) \cap \QQ[\xxx_n,\yyy_d]^{\symm_d})
\end{equation}
as $\QQ$-vector spaces. The following result is a relationship between bases of these vector spaces.

\begin{lemma}
    \label{lem:vector-space-relationship}
    The $\QQ$-vector spaces in \eqref{eq:two-vector-spaces} both have dimension $\frac{k!}{(k-d)!} \cdot \Stir(n,d)$. Furthermore, if $\BBB_{n,k,d} \subseteq \QQ[\xxx_n,\yyy_d]^{\symm_d}$ is a family of $\symm_d$-invariant homogeneous polynomials which descends to a vector space basis of $\QQ[\xxx_n,\yyy_d]^{\symm_d}/(J_{n,k,d}^\QQ \cap \QQ[\xxx_n,\yyy_d]^{\symm_d})$, then $\BBB_{n,k,d}$ also descends to a vector space basis of $\QQ[\xxx_n,\yyy_d]^{\symm_d}/(\II(\Zpoints_{n,k,d}(\aalpha_k)) \cap \QQ[\xxx_n,\yyy_d]^{\symm_d})$.
\end{lemma}

\begin{proof}
    The Reynolds operator $(-)^\natural = \frac{1}{d!} \sum_{w \in \symm_d} w \cdot (-)$ gives a projection $\QQ[\xxx_n,\yyy_d] \twoheadrightarrow \QQ[\xxx_n,\yyy_d]^{\symm_d}$ which induces an isomorphism
    \begin{equation}
        (\QQ[\xxx_n,\yyy_d]/J_{n,k,d}^\QQ)^{\symm_d} \xrightarrow{\, \, \sim \, \, } \QQ[\xxx_n,\yyy_d]^{\symm_d}/(J_{n,k,d}^\QQ \cap \QQ[\xxx_n,\yyy_d]^{\symm_d})
    \end{equation}
    of graded $\QQ$-vector spaces and an isomorphism
    \begin{equation}
        (\QQ[\xxx_n,\yyy_d]/\II(\Zpoints_{n,k,d}(\aalpha_k)))^{\symm_d} \xrightarrow{\, \, \sim \, \,}
        \QQ[\xxx_n,\yyy_d]^{\symm_d}/(\II(\Zpoints_{n,k,d}(\aalpha_k)) \cap \QQ[\xxx_n,\yyy_d]^{\symm_d})
    \end{equation}
    of ungraded $\QQ$-vector spaces. By Proposition~\ref{prop:orbit-harmonics-Z-homogeneous}, all four of these $\QQ$-vector spaces have dimension equal to the number of $\symm_d$-orbits in $\Zpoints_{n,k,d}(\aalpha_k)$. Equation~\eqref{eq:z-orbit-count} implies that this dimension is $\frac{k!}{(k-d)!} \cdot \Stir(n,d)$.

    Lemma~\ref{lem:invariant-ideal-generators} implies that 
    \begin{quote}$J_{n,k,d}^\QQ \cap \QQ[\xxx_n,\yyy_d]^{\symm_d}$ is the associated graded ideal of $\II(\Zpoints_{n,k,d}(\aalpha_k)) \cap \QQ[\xxx_n,\yyy_d]^{\symm_d}$ within the invariant ring $\QQ[\xxx_n,\yyy_d]^{\symm_d}$.
    \end{quote}
    With this algebraic fact in hand, the second part of the lemma follows from the following general fact. Let $R = \bigoplus_{d \geq 0} R_d$ be any graded $\QQ$-algebra. If $I \subseteq R$ is an ideal, we may define the associated graded ideal $\gr \, I \subseteq R$ as in Section~\ref{sec:Background}.

    {\bf Claim:} {\em  Let $\BBB$ be a set of homogeneous elements of $R$ which descends to a $\QQ$-basis of $R/\gr \, I$. Then $\BBB$ descends to a basis of $R/I.$}

    This claim and its (standard) proof are very similar to those of Lemma~\ref{lem:graded-basis}; we include an argument for completeness. If $\BBB$ did not span $R/I$, there would exist a homogeneous element $g \in R_d$ of minimal degree $d$ such that $g \notin \mathrm{span}_\QQ(\BBB) \mod I$. Since $\BBB$ spans $R/\gr \, I$, there exist rational numbers $c_b \in \QQ$ and an element $f \in I$ so that
    \begin{equation}
        g = \sum_{b \, \in \, \BBB} c_b \cdot b + \tau(f)
    \end{equation}
    where $\tau(f)$ is the top degree homogeneous component of $f$. Since both $g$ and the elements of $\BBB$ are homogeneous and $g \notin \mathrm{span}_\QQ(\BBB) \mod I$, we deduce that $\deg(f) = \deg(g) = d$. The minimality of $d$ implies $\tau(f) - f \in \mathrm{span}_\QQ(\BBB) \mod I$, so 
    \begin{equation}
        g = \sum_{b \, \in \, \BBB} c_b \cdot b + (\tau(f) - f) + f \in \mathrm{span}_\QQ(\BBB) \mod I,
    \end{equation}
    a contradiction which shows that $\BBB$ spans $R/I$. 

    If $\BBB$ were linearly dependent modulo $I$, there would exist an element $f \in I$ and scalars $c_b \in \QQ$ not all zero so that
    \begin{equation}
        \sum_{b \, \in \, \BBB} c_b \cdot b = f.
    \end{equation}
    Since $\BBB$ is linearly independent modulo $\gr \, I$, we know that $\BBB$ is linearly independent as a subset of $R$ and $f \neq 0$. If $\deg(f) = d$, since elements of $\BBB$ are homogeneous we have
    \begin{equation}
        \sum_{\deg(b) \, = \, d} c_b \cdot b = \tau(f) \in \gr \, I.
    \end{equation}
    The linear independence of $\BBB$ in $R$ forces $c_b \neq 0$ for at least one $b \in \BBB$ with $\deg(b) = d$, so $\BBB$ is linearly dependent modulo $R/\gr \, I$, a contradiction which completes the proof of the claim and the lemma.
\end{proof}

\subsection{An infinite family} In order to incorporate the torus variables $\ttt_k$, we put the loci $\Zpoints_{n,k,d}(\aalpha_k)$ into an infinite family as follows. 

\begin{defn}
    \label{def:z-family}
    Fix positive integers $n \geq k \geq d$. We let $\Zpoints_{n,k,d} \subseteq \QQ^{n+d+k}$ be the set of points $(z_1, \dots, z_n; z_{n+1}, \dots, z_{n+d}; z_{n+d+1}, \dots, z_{n+d+k})$ such that
    \begin{itemize}
        \item there exists an injective function $f: \{n+1,\dots,n+d\} \hookrightarrow \{n+d+1,\dots,n+d+k\}$ such that $z_i = z_{f(i)}$ for all $n+1 \leq i \leq n+d$, and 
        \item there exists a surjective function $g: \{1,\dots,n\} \twoheadrightarrow \{n+1,\dots,n+d\}$ such that $z_j = z_{g(j)}$ for all $1 \leq j \leq n$.
    \end{itemize}
\end{defn}

If $\pi: \QQ^{n+d+k} \twoheadrightarrow \QQ^{n+d}$ is the projection which forgets the last $k$ coordinates, we have 
\begin{equation}
\label{eq:pi-fiber}
    \pi^{-1}(\aalpha_k) \cap \Zpoints_{n,k,d} = \Zpoints_{n,k,d}(\aalpha_k) \times \{\aalpha_k\}
\end{equation}
whenever $\aalpha_k \in \QQ^k$ has distinct coordinates. It is not hard to see that within $\QQ^{n+d+k}$ we have
\begin{equation}
    \Zpoints_{n,k,d} = \text{Zariski closure of }
    \bigcup_{\aalpha_k} \Zpoints_{n,k,d}(\aalpha_k) \times \{\aalpha_k\}
\end{equation}
where the union is over all $\aalpha_k \in \QQ^k$ with distinct coordinates.

Let $\QQ[\xxx_n,\yyy_d,\ttt_k]$ be the coordinate ring of $\QQ^{n+d+k}$. The vanishing ideal $\II(\Zpoints_{n,k,d})$ is a subset of $\QQ[\xxx_n,\yyy_d,\ttt_k]$. We will prove that the ideal $\II(\Zpoints_{n,k,d})$ has the following set of generators.

\begin{defn}
    \label{def:j-t-ideal}
    Let $J_{n,k,d}^{\QQ,\ttt} \subseteq \QQ[\xxx_n,\yyy_d,\ttt_k]$ be the ideal generated by 
    \begin{itemize}
    \item $e_r(\ttt_k) - e_{r-1}(\ttt_k) h_1(\yyy_d) + \cdots + (-1)^r h_r(\yyy_d)$ for all $r > k-d$,
    \item $e_r(\xxx_n) - e_{r-1}(\xxx_n) h_1(\yyy_d) + \cdots + (-1)^r h_r(\yyy_d)$ for all $r > n-d$, and
    \item $x_i^d - x_i^{d-1} e_1(\yyy_d) + \cdots + (-1)^d e_d(\yyy_d)$ for $i = 1, \dots , n$.
    \end{itemize}
\end{defn}

Observe that the generators of $J_{n,k,d}^\QQ \subseteq \QQ[\xxx_n,\yyy_d]$ in Definition~\ref{def:j-ideal-definition} are obtained from those of $J_{n,k,d}^{\QQ,\ttt} \subseteq \QQ[\xxx_n,\yyy_d,\ttt_k]$ in Definition~\ref{def:j-t-ideal} by setting the $t$-variables equal to 0. Also, the generators of $\II(\Zpoints_{n,k,d}(\aalpha_k))$ in Lemma~\ref{lem:orbit-harmonics-Z-inhomogeneous} are obtained from those in Definition~\ref{def:j-t-ideal} by setting $t_i \to \alpha_i$. We start by establishing the following containment of ideals.

\begin{lemma}
    \label{lem:j-t-containment}
    We have the containment of ideals $J_{n,k,d}^{\QQ,\ttt} \subseteq \II(\Zpoints_{n,k,d})$.
\end{lemma}

\begin{proof}
    This is similar to the proof of Lemma~\ref{lem:orbit-harmonics-Z-inhomogeneous}. Let $u$ be a new variable. If $r > k-d$, the rational function
    $$ \frac{(1 + t_1 u) \cdots (1 + t_k u)}{(1 + y_1 u) \cdots (1 + y_d u)}$$
    when restricted to $\Zpoints_{n,k,d}$ is a polynomial in $u$ of degree $k-d$. Since the coefficient of $u^r$ in this rational function is $e_r(\ttt_k) - e_{r-1}(\ttt_k) h_1(\yyy_d) + \cdots + (-1)^r h_r(\yyy_d)$, we see that the generators in the first bullet point of Definition~\ref{def:j-t-ideal} lie in $\II(\Zpoints_{n,k,d})$. The generators in the second and third bullet points are handled in a similar fashion.
\end{proof}

We upgrade Lemma~\ref{lem:j-t-containment} to an equality of ideals. In the process of doing this, we will establish a vector space basis of the corresponding quotient ring. This basis will have the following form.

\begin{defn}
    \label{def:t-augmentation}
    For any set $\BBB$ of polynomials, let $\BBB^\ttt$ be the set of polynomials given by
    \begin{equation}
        \BBB^\ttt := \{ f \cdot t_1^{a_1} \cdots t_k^{a_k} \,:\, f \in \BBB, \, a_1, \dots, a_k \geq 0 \}.
    \end{equation}
\end{defn}

 Recall the set $\AAA_{n,k,d}$ of monomials in $\QQ[\xxx_n,\yyy_d]$ of Definition~\ref{defn:a-monomials}. The next result shows both that $\II(\ZZZ_{n,k,d}) = J_{n,k,d}^{\QQ,\ttt}$ and that $\AAA_{n,k,d}^\ttt$ descends to a basis of $\QQ[\xxx_n,\yyy_d,\ttt_k]$ modulo this common ideal.

\begin{proposition}
    \label{prop:t-ideal-equality}
    We have the equality of ideals $J_{n,k,d}^{\QQ,\ttt} = \II(\Zpoints_{n,k,d})$ inside $\QQ[\xxx_n,\yyy_d,\ttt_k]$. The set $\AAA_{n,k,d}^\ttt$ descends to a $\QQ$-basis for the common quotient ring
    $$\QQ[\xxx_n,\yyy_d,\ttt_k]/J_{n,k,d}^{\QQ,\ttt} = \QQ[\xxx_n,\yyy_d,\ttt_k]/\II(\Zpoints_{n,k,d}).$$
    In particular, this quotient is a free $\QQ[\ttt_k]$-module of rank $\frac{k! \cdot d!}{(k-d)!} \cdot \Stir(n,d)$.
\end{proposition}

\begin{proof}
    The last sentence follows from the first two. We start by establishing the linear independence of $\AAA_{n,k,d}^\ttt$ modulo $\II(\Zpoints_{n,k,d})$. For any $\aalpha_k = (\alpha_1, \dots, \alpha_n) \in \QQ^k$, let 
    \begin{equation}
        \varepsilon_{\aalpha_k}: \QQ[\xxx_n,\yyy_d,\ttt_k] \twoheadrightarrow \QQ[\xxx_n,\yyy_d]
    \end{equation}
    be the evaluation homomorphism which sends $t_i$ to $\alpha_i$. It follows from \eqref{eq:pi-fiber} that 
    \begin{quote}
        whenever $\aalpha_k \in \QQ^k$ has distinct coordinates, the image of $\II(\Zpoints_{n,k,d})$ under $\varepsilon_{\aalpha_k}$ is a subset of $\II(\Zpoints_{n,k,d}(\aalpha_k))$.
    \end{quote}
    If $\AAA_{n,k,d}^\ttt$ were not linearly independent modulo $\II(\Zpoints_{n,k,d})$, there would exist polynomials $f_a(\ttt_k) \in \QQ[\ttt_k]$ not all zero such that
    \begin{equation}
        \sum_{a \, \in \, \AAA_{n,k,d}} f_a(\ttt_k) \cdot a \in \II(\Zpoints_{n,k,d}).
    \end{equation}
    If $\aalpha_k \in \QQ^k$ has distinct points, we may apply the evaluation homomorphism $\varepsilon_{\aalpha_k}$ to get
    \begin{equation}
        \sum_{a \, \in \, \AAA_{n,k,d}} f_a(\aalpha_k) \cdot a \in \II(\Zpoints_{n,k,d}(\aalpha_k)).
    \end{equation}
    By Lemma~\ref{lem:graded-basis} and Proposition~\ref{prop:orbit-harmonics-Z-homogeneous}, the set $\AAA_{n,k,d}$ is linearly independent modulo $\II(\Zpoints_{n,k,d}(\aalpha_k))$. We conclude that
    \begin{quote}
        for all $a \in \AAA_{n,k,d}$ we have $f_a(\aalpha_k) = 0$ whenever $\aalpha_k \in \QQ^k$ has distinct coordinates.
    \end{quote}
    Since the set of points in $\QQ^k$ with distinct coordinates is dense in $\QQ^k$, this forces $f_a(\ttt_k) \equiv 0$ for each $a \in \AAA_{n,k,d}$. We conclude that $\AAA_{n,k,d}^\ttt$ is linearly independent modulo $\II(\Zpoints_{n,k,d}(\aalpha_k))$.

    By Lemma~\ref{lem:j-t-containment}, the proposition will be proven if we can show that $\AAA_{n,k,d}^\ttt$ spans $\QQ[\xxx_n,\yyy_d,\ttt_k]$ modulo $J_{n,k,d}^{\QQ,\ttt}$. To this end, let $m = m_\xxx \cdot m_\yyy \cdot m_\ttt$ be a monomial in $\QQ[\xxx_n,\yyy_d,\ttt_k]$ where $m_\xxx$ involves the $x$-variables, $m_\yyy$ involves the $y$-variables, and $m_\ttt$ involves the $t$-variables. We want to show that $m$ lies in the span of $\AAA_{n,k,d}^\ttt$ modulo $J_{n,k,d}^{\QQ,\ttt}$. If $\deg(m_\xxx \cdot m_\yyy) = 0$, we have $m_\ttt \in \AAA_{n,k,d}^\ttt$ and this is clear, so assume that $\deg(m_\xxx \cdot m_\yyy) > 0$. If $g$ is a generator of $J_{n,k,d}^{\QQ,\ttt}$, let $\bar{g}$ be the corresponding generator of $J_{n,k,d}^\QQ$ obtained by setting the $t$-variables equal to 0. By Proposition~\ref{prop:orbit-harmonics-Z-homogeneous} we may write 
    \begin{equation}
    \label{eq:ideal-equality-one}
        m_\xxx \cdot m_\yyy = \sum_{a \, \in \, \AAA_{n,k,d}} \gamma_a \cdot a + \sum_g h_g \cdot \bar{g}
    \end{equation}
    for some scalars $\gamma_a \in \QQ$ and polynomials $h_g \in \QQ[\xxx_n,\yyy_d]$ where the sum is over generators $g$ of $J_{n,k,d}^{\QQ,\ttt}$. In particular, Equation~\eqref{eq:ideal-equality-one} does not involve any $t$-variables. Discarding redundant terms, we may assume that 
    \begin{itemize}
    \item
    $\gamma_a \neq 0$ whenever $\deg(m_\xxx \cdot m_\yyy) \neq \deg(a)$ and 
    \item every product $h_g \cdot g'$ is homogeneous of degree $\deg(m_\xxx \cdot m_\yyy)$. 
    \end{itemize}
    Multiplying through by $m_\ttt$ gives
    \begin{multline}
        m = m_\xxx \cdot m_\yyy \cdot m_\ttt=  \sum_{a \, \in \, \AAA_{n,k,d}} \gamma_a \cdot a \cdot m_\ttt + \sum_{g} h_g \cdot \bar{g} \cdot m_\ttt \\ 
        = \sum_{a \, \in \, \AAA_{n,k,d}} \gamma_a \cdot a \cdot m_\ttt + \sum_{g} h_g \cdot (\bar{g}-g) \cdot m_\ttt + \sum_{g} h_g \cdot g \cdot m_\ttt.
    \end{multline}
    The $x$-degree plus the $y$-degree of each term $h_g \cdot (\bar{g}-g) \cdot m_\ttt$ in the second sum is strictly less than $\deg(m_\xxx \cdot m_\yyy)$. Since $\sum_{g} h_g \cdot g \cdot m_\ttt \in J_{n,k,d}^{\QQ,\ttt}$, by induction on $\deg(m_\xxx \cdot m_\yyy)$ we conclude that $m$ lies in the span of $\AAA_{n,k,d}^\ttt$ modulo $J_{n,k,d}^{\QQ,\ttt}$.
    \end{proof}

    The symmetric group $\symm_d$ acts on the middle variables of $\QQ[\xxx_n,\yyy_d,\ttt_k]$. The following result has the same proof as Lemma~\ref{lem:invariant-ideal-generators} involving the Reynolds operator; we omit the argument this time.

    \begin{lemma}
        \label{lem:t-invariant-ideal-generators}
        As an ideal in $\QQ[\xxx_n,\yyy_d,\ttt_k]^{\symm_d}$, the intersection
        $$\II(\Zpoints_{n,k,d}) \cap \QQ[\xxx_n,\yyy_d,\ttt_k]^{\symm_d} = J_{n,k,d}^{\QQ,\ttt} \cap \QQ[\xxx_n,\yyy_d,\ttt_k]^{\symm_d}$$
        has the same generating set as that of $J_{n,k,d}^{\QQ,\ttt}$ in Definition~\ref{def:j-t-ideal}.
    \end{lemma}

    As in the setting without $t$-variables, the Reynolds operator induces an isomorphism of $\QQ[\ttt_k]$-modules
    \begin{equation}
     (\QQ[\xxx_n,\yyy_d,\ttt_k]/J_{n,k,d}^{\QQ,\ttt})^{\symm_d} \xrightarrow{ \, \, \sim \, \, }
     \QQ[\xxx_n,\yyy_d,\ttt_k]^{\symm_d}/(J_{n,k,d}^{\QQ,\ttt} \cap \QQ[\xxx_n,\yyy_d,\ttt_k]^{\symm_d}).
    \end{equation}
    The next result implies that this common quotient is a free $\QQ[\ttt_k]$-module.

    \begin{lemma}
        \label{lem:t-invariant-free-module}
        Let $\BBB_{n,k,d} \subseteq \QQ[\xxx_n,\yyy_d]^{\symm_d}$ be a family of $\symm_d$-invariant homogeneous polynomials which descends to a $\QQ$-basis of $\QQ[\xxx_n,\yyy_d]^{\symm_d}/(J_{n,k,d}^\QQ \cap \QQ[\xxx_n,\yyy_d]^{\symm_d})$. Then $\BBB_{n,k,d}^\ttt$ descends to a $\QQ$-basis of $\QQ[\xxx_n,\yyy_d,\ttt_k]^{\symm_d}/(J_{n,k,d}^{\QQ,\ttt} \cap \QQ[\xxx_n,\yyy_d,\ttt_k]^{\symm_d})$. 
        
        In particular, the quotient $\QQ[\xxx_n,\yyy_d,\ttt_k]^{\symm_d}/(J_{n,k,d}^{\QQ,\ttt} \cap \QQ[\xxx_n,\yyy_d,\ttt_k]^{\symm_d})$ is a free $\QQ[\ttt_k]$-module of rank $\frac{k!}{(k-d)!} \cdot \Stir(n,d)$.
    \end{lemma}

    \begin{proof}
        The second part of the lemma follows from the first part and Lemma~\ref{lem:vector-space-relationship}. The first part is proven in a fashion similar to Proposition~\ref{prop:t-ideal-equality}; we only sketch the argument this time.

        One first shows that $\BBB_{n,k,d}^\ttt \subseteq \QQ[\xxx_n,\yyy_d,\ttt_k]^{\symm_d}$ is linearly independent modulo 
        $$J_{n,k,d}^{\QQ,\ttt} \cap \QQ[\xxx_n,\yyy_d,\ttt_k]^{\symm_d} = \II(\Zpoints_{n,k,d}) \cap \QQ[\xxx_n,\yyy_d,\ttt_k]^{\symm_d}.$$
        However, if $f_b(\ttt_k) \in \QQ[\ttt_k]$ are polynomials such that $\sum_{b \in \BBB_{n,k,d}} f_b(\ttt_k) \cdot b \in \II(\Zpoints_{n,k,d}) \cap \QQ[\xxx_n,\yyy_d,\ttt_k]^{\symm_d}$, the argument of Proposition~\ref{prop:t-ideal-equality} involving the evaluation homomorphisms $\varepsilon_{\aalpha_k}$ shows  that $f_b(\ttt_k) \equiv 0$ for all $b \in \BBB_{n,k,d}$.

        Next, we claim that $\BBB_{n,k,d}^\ttt$ spans $\QQ[\xxx_n,\yyy_d,\ttt_k]^{\symm_d}$ modulo $J_{n,k,d}^{\QQ,\ttt} \cap \QQ[\xxx_n,\yyy_d,\ttt_k]^{\symm_d}$. The ring $\QQ[\xxx_n,\yyy_d,\ttt_k]^{\symm_d}$ is spanned by elements of the form $f(\xxx_n,\yyy_d) \cdot m_\ttt$
        where $f(\xxx_n,\yyy_d) \in \QQ[\xxx_n,\yyy_d]^{\symm_d}$ is homogeneous and $m_\ttt$ is a monomial in the $t$-variables. If $\deg f = 0$ then $f(\xxx_n,\yyy_d) \cdot m_\ttt$ clearly lies in the span of $\BBB_{n,k,d}^\ttt$. If $\deg f > 0$, Lemma~\ref{lem:t-invariant-ideal-generators} gives a one-to-one relationship between generators $g$ of $J_{n,k,d}^{\QQ,\ttt} \cap \QQ[\xxx_n,\yyy_d,\ttt_k]^{\symm_d}$ and generators $\bar{g}$ of $J_{n,k,d}^\QQ \cap \QQ[\xxx_n,\yyy_d]^{\symm_d}$ where $\bar{g}$ is obtained from $g$ by setting the $t$-variables equal to 0. One repeats the inductive argument in the proof of Proposition~\ref{prop:t-ideal-equality} to see that $f(\xxx_n,\yyy_d) \cdot m_\ttt$ lies in the span of $\BBB_{n,k,d}^\ttt$ modulo $J_{n,k,d}^{\QQ,\ttt} \cap \QQ[\xxx_n,\yyy_d,\ttt_k]^{\symm_d}$.
    \end{proof}

\section{Equivariant Cohomology Presentation}
\label{sec:Cohomology}

In this section and the next, we will need some notation related to words over the positive integers. For $d \leq k \leq n$, we write
\begin{equation}
    \WWW_{n,k,d} := \{ w: [n] \to [k] \,:\, |\im(w)| = d \}
\end{equation}
for the set of length $n$ words $[w(1), \dots, w(n)]$ over the alphabet $[k]$ in which exactly $d$ letters appear. As in Section~\ref{sec:Background}, we have
\begin{equation}
    |\WWW_{n,k,d}| = \frac{k!}{(k-d)!} \cdot \Stir(n,d).
\end{equation}
For example, the set $\WWW_{3,3,2}$ consists of the words
\begin{multline*}
[1,1,2], \, [1,2,1], \, [2,1,1], \, [1,2,2], \, [2,1,2], \, [2,2,1], \\ 
[1,1,3], \, [1,3,1], \, [3,1,1], \, [1,3,3], \, [3,1,3], \, [3,3,1], \\
[2,2,3], \, [2,3,2], \, [3,2,2], \, [2,3,3], \, [3,2,3], \, [3,3,2].
\end{multline*}
When $d = k$, we have the set
\begin{equation}
    \WWW_{n,k} := \WWW_{n,k,k}
\end{equation}
of {\em Fubini words}, or surjective maps $w: [n] \twoheadrightarrow [k]$.

\subsection{A fiber bundle} Our geometric analysis of $X_{n,k,d}$ makes heavy use of the fact that $X_{n,k,d}$ is a fiber bundle over the Grassmannian $\Gr(d,\CC^k)$. The fibers of $X_{n,k,d} \to \Gr(d,\CC^k)$ are isomorphic to the following variety.

\begin{defn}
    \label{def:xnd}
    For $n \geq d$, let $X_{n,d}$ be the moduli space of $n$-tuples $(\ell_1, \dots, \ell_n)$ of lines in $\CC^d$ which satisfy $\ell_1 + \cdots + \ell_n = \CC^d$.
\end{defn}

The space $X_{n,d}$ was first defined in \cite{PR}; we have $X_{n,d,d} = X_{n,d}$. The general relationship between the spaces $X_{n,k,d}$ and $X_{n,d}$ is as follows.

\begin{observation}
    \label{obs:fiber-bundle} The map $p: X_{n,k,d} \to \Gr(d,\CC^k)$ sending $(\ell_1, \dots, \ell_n)$ to $\ell_1 + \cdots + \ell_n$ is a fiber bundle with fiber isomorphic to $X_{n,d}$.
\end{observation}

The rank $k$ torus $T = (\CC^*)^k$ acts on both $X_{n,k,d}$ and $\Gr(d,\CC^k)$.
Observation~\ref{obs:fiber-bundle} and the Leray-Hirsch Theorem have the following consequences for ordinary and equivariant cohomology.

\begin{lemma}
    \label{lem:leray-hirsch-decomposition}
    We have an isomorphism of $H^*(\Gr(d,\CC^k);\ZZ)$-modules
    $$H^*(X_{n,k,d};\ZZ) \cong H^*(\Gr(d,\CC^k);\ZZ) \otimes_\ZZ H^*(X_{n,d};\ZZ)$$
    and an isomorphism of $H^*_T(\Gr(d,\CC^k);\ZZ)$-modules
    $$H^*_T(X_{n,k,d};\ZZ) \cong H^*_T(\Gr(d,\CC^k);\ZZ) \otimes_\ZZ H^*(X_{n,d};\ZZ)$$
\end{lemma}

\begin{proof}
 First we show the non-equivariant case. It suffices to check that the two conditions in the Leray-Hirsch Theorem (Theorem \ref{thm:leray-hirsch}) are satisfied. Pawlowski and Rhoades \cite{PR} obtained a family of explicitly defined polynomials $\mathfrak{S}_w(\mathbf{x}_n) \in \mathbb{Z}_{\geq 0}[\mathbf{x}_n]$ for Fubini words $w\in \mathcal{W}_{n,d}$ which descend to a homogeneous $\ZZ$-basis of the quotient ring
 \begin{equation}
     \ZZ[\xxx_n]/(x_1^d, x_2^d, \dots, x_n^d, e_n(\xxx_n), e_{n-1}(\xxx_n), \dots, e_{n-d+1}(\xxx_n))
 \end{equation}
 which is canonically identified with the cohomology ring $H^*(X_{n,d};\ZZ)$. Here in $H^*(X_{n,d};\ZZ)$ we identify $x_i$ with the first Chern class of the dual of the $i$-th tautological line bundle over $X_{n,d}$. In particular, we know that $H^*(X_{n,d};\ZZ)$ is a free $\ZZ$-module so that the first condition is satisfied. Let $\mathcal{L}_i$ be the $i$-th tautological line bundle over $X_{n,k,d}$ and $b_i:=c_1(\mathcal{L}_i)$. Consider the classes $\sigma_w, w\in \mathcal{W}_{n,d}$ in $H^*(X_{n,k,d};\ZZ)$ given by $\sigma_w = \mathfrak{S}_w(\mathbf{x}_n)|_{\mathbf{x}_n=-\mathbf{b}_n}$ where $\mathbf{b}_n=(b_1,\dots, b_n)$. 

Let $X_n(U_d)$ be the fiber of the projection $X_{n,k,d} \to \Gr(d,k)$ at each $U_d \in \Gr(d,k)$. The space $X_n(U_d)$ is a subspace of the $n$-fold product $\PP(U_d) \times \cdots \times \PP(U_d)$ that fits in the following fiber diagram
\[
\xymatrix{
\PP(U_d) \times \cdots \times \PP(U_d) \ar[r] & \PP(\CC^k) \times \cdots \times \PP(\CC^k)&\\
X_n(U_d) \ar[r]_f\ar[u]_\iota & X_{n,k,d}  \ar[u] \ar[r]_p & \Gr(d,k)
}
\] 
where $f$ is the inclusion of the fiber.  The pullback $f^*(\mathcal{L}_i)$ of $\mathcal{L}_i$ is the $i$-tautological line bundle over $X_n(U_d)$ with respect to the natural inclusion denoted by $\iota$ in the diagram.  Thus we have
\[
f^*(\sigma_w) = f^*(\mathfrak{S}_w(-b_1,\dots, -b_n))=\mathfrak{S}_w(c_1(f^*(\mathcal{L}_1^*)),\dots, c_1(f^*(\mathcal{L}_n^*))),  
\]
and see that $f^*(\sigma_w), w\in \mathcal{W}_{n,d}$ form a homogeneous basis of $H^*(X_n(U_d);\ZZ)$ by the fact mentioned in the first paragraph in this proof. Thus the second condition for the Leray-Hirsch Theorem holds.

For the equivariant setting, we consider the fiber bundle enhanced by the Borel construction
\[
X_n(U_d) \stackrel{f}{\to} ET \times_TX_{n,k,d} \to ET\times_T\Gr(d,k).  
\]
Note that the fiber is still $X_n(U_d)$ and we denote the inclusion of the fiber by $f$ again. We replace $b_i$ with the $T$-equivariant Chern class $c_1^T(\mathcal{L}_i)$ and consider the $T$-equivariant classes $\sigma_w$. The pullbacks $f^*\sigma_w, w\in \mathcal{W}_{n,d}$ form a $\ZZ$-basis of $H^*(X_n(U_d);\ZZ)$ so that the second condition still holds in the equivariant setting. This completes the proof.
\end{proof}

We want to describe the $\ZZ[\ttt_k]$-module structure of $H^*_T(X_{n,k,d};\ZZ)$. To do this, we recall a geometric result of Pawlowski and Rhoades \cite{PR}. Let $\MMM_{n,k}$ be the space of $k \times n$ complex matrices with no zero columns. The space $\MMM_{n,k}$ carries a left action of $T = (\CC^*)^k$ by row scaling and a right action of $(\CC^*)^n$ by column scaling.  There is a natural identification
\begin{equation}
    \MMM_{n,k}/(\CC^*)^n = (\PP^{k-1})^n
\end{equation}
given by associating a matrix with columns $A = \begin{pmatrix} v_1 & \cdots & v_n \end{pmatrix}$ to the tuple of lines $(\ell_1, \dots, \ell_n)$ where $\ell_i$ is spanned by $v_i$.

 Let $w: [n] \to [k]$ be a word. A position $1 \leq j \leq n$ is called {\em initial} if $w(j') \leq w(j)$ for $j < j'$. Write 
\begin{equation}
    \initial(w) := \{ 1 \leq j \leq n \,:\, \text{$j$ is initial for $w$} \}
\end{equation}
for the initial positions of $w$.  For example, we have $\initial(21231) = \{1,2,4\}$. The {\em pattern matrix} $\PM(w)$ is the $k \times n$ matrix over $\{0,1,\star\}$ with entries $\PM(w)_{i,j}$ given as follows.
\begin{itemize}
    \item We have $\PM(w)_{i,j} = 1$ if and only if $w(j) = i$.
    \item Suppose $j \in \initial(w)$ and $w(j) \neq i$. If $w(j) > i$ and there exists $j' < j$ with $w(j') = i$ then $\PM(w)_{i,j} = \star$. Otherwise $\PM(w)_{i,j} = 0$.
    \item Suppose $j \notin \initial(w)$ and $w(j) \neq i$. If the first occurrence of $i$ in $[w(1), \dots, w(n)]$ is before the first occurrence of $w(j)$, then $\PM(w)_{i,j} = \star$. Otherwise $\PM(w)_{i,j}= 0$.
\end{itemize}
For example, if $n = 6$ and $k = 4$ we have
$$\PM(441422) = 
\begin{pmatrix} 
0 & 0 & 1 & 0 & \star & \star \\
0 & 0 & 0 & 0 & 1 & 1 \\
0 & 0 & 0 & 0 & 0 & 0 \\
1 & 1 & 0 & 1 & 0 & \star
\end{pmatrix}.$$
Let $U(w) \subseteq GL_k(\CC)$ be the group of lower triangular matrices with 1's on the diagonal and 0's in any off-diagonal positions of a column $j$ with $j \not\in \im(w)$. In the above example we have
$$U(441422) = \left\{ \begin{pmatrix} 1 & 0 & 0 & 0 \\ * & 1  & 0 & 0 
\\ * & *& 1 & 0 \\ * & * & 0 & 1 \end{pmatrix} \right\}$$
where the $*$'s are complex numbers.

For any word $w: [n] \to [k]$, let $\widehat{C}_w \subseteq \MMM_{n,k}$ be the affine space of matrices obtained by replacing the $\star$'s in $\PM(w)$ with complex numbers. We define $C_w \subseteq (\PP^{k-1})^n$ by
\begin{equation}
    C_w := U(w) \cdot  \widehat{C}_w \cdot (\CC^*)^n / (\CC^*)^n.
\end{equation}
Observe that $C_w \subseteq X_{n,k,d}$ if and only if $w \in \WWW_{n,k,d}$. In fact, we have the following result, essentially proven in \cite{PR}.

\begin{lemma}
    \label{lem:affine-paving}
    The set $\{C_w \,:\, w \in \WWW_{n,k,d} \}$ forms the cells of an affine paving of $X_{n,k,d}$. The cells $C_w$ are closed under the left action of $T = (\CC^*)^k$.
\end{lemma}

\begin{proof}
    It follows from the definitions that $C_w$ is closed under the left action of $T$ by row scaling. The fact that $C_w$ is isomorphic to an affine space is \cite[Lem. 5.6]{PR}. In \cite[Lem. 5.7]{PR} it is proven that the collection of cells $\{C_w \,:\, w: [n] \to [k] \}$ indexed by all words $w: [n] \to [k]$ gives an affine paving of $(\PP^{k-1})^n$. Furthermore, the affine paving constructed in the proof of \cite[Lem. 5.7]{PR} has an order of cells which refines the partial order 
    \begin{center}$w \prec w'$ if and only if $|\im(w)| < |\im(w')| $.
\end{center}The set of cells  $\{C_w \,:\, w \in \WWW_{n,k,d} \}$ corresponding to words with image size $d$ therefore gives an affine paving of $X_{n,k,d}$.
\end{proof}

Lemma~\ref{lem:affine-paving-cohomology} and the above result imply that the cell closures $\{\overline{C}_w \,:\, w \in \WWW_{n,k,d}\}$ induce a $\ZZ$-basis for $H^*(X_{n,k,d};\ZZ)$. Similarly, by Lemma~\ref{lem:affine-paving-equivariant} the equivariant classes of these cell closures are a $\ZZ[\ttt_k]$-basis for $H^*_T(X_{n,k,d};\ZZ)$. Since the set $\WWW_{n,k,d}$ has 
$\frac{k!}{(k-d)!} \cdot \Stir(n,d)$ elements, the following result is immediate.

\begin{lemma}
    \label{lem:cohomology-ring-ranks} Let $d \leq k \leq n$.
    The ordinary cohomology ring $H^*(X_{n,k,d};\ZZ)$ is a free $\ZZ$-module of rank $\frac{k!}{(k-d)!} \cdot \Stir(n,d)$. The equivariant cohomology ring $H^*_T(X_{n,k,d};\ZZ)$ is a free $\ZZ[\ttt_k]$-module of rank $\frac{k!}{(k-d)!} \cdot \Stir(n,d)$.
\end{lemma}

%\begin{proof}
%Recall the classical {\em Schubert paving} of $\Gr(d,\CC^k)$. A point of $\Gr(d,\CC^k)$ may be represented by a full rank $d \times k$ matrix which has a unique reduced row echelon form. This gives rise to a disjoint union decomposition
%\begin{equation}
%    \Gr(d,\CC^k) = \bigsqcup_\lambda \mathring{X}_\lambda
%\end{equation}
%where the disjoint union is over partitions $\lambda$ which fit inside a $d \times (k-d)$ box and
%$\mathring{X}_\lambda \cong \CC^{d(k-d) - |\lambda|}$ is the open Schubert cell corresponding to RREFs of sparsity pattern $\lambda$. These cells fit together to give an affine paving of $\Gr(d,\CC^k)$. By Lemma~\ref{lem:affine-paving-cohomology} $H^*(\Gr(d,\CC^k);\ZZ)$ is a free $\ZZ$-module of rank ${k \choose d}$. Since the cells $\mathring{X}_\lambda$ are stable under the $T$-action on columns, Lemma~\ref{lem:affine-paving-equivariant} implies that $H^*_T(\Gr(d,\CC^k);\ZZ)$ is a free $\ZZ[\ttt_k]$-module of rank ${k \choose d}$.

%It was proved by Pawlowski and Rhoades \cite{PR} that $X_{n,d}$ admits an affine paving with cells indexed by surjective functions $[n] \twoheadrightarrow [d]$. This paving was studied further by Billey and Ryan \cite{BR}. The ordinary cohomology ring $H^*(X_{n,d};\ZZ)$ is therefore a free $\ZZ$-module of rank $d! \cdot \Stir(n,d)$. The result follows from Lemma~\ref{lem:leray-hirsch-decomposition}.
%\end{proof}

The fiber bundle $p: X_{n,k,d} \to \Gr(d,\CC^k)$ in Observation~\ref{obs:fiber-bundle} gives a generating set of $H^*_T(X_{n,k,d};\ZZ)$ as follows. Let $\UUU_d$ be the rank $d$ tautological vector bundle over $\Gr(d,\CC^k)$ whose fiber over a subspace $V \in \Gr(d,\CC^k)$ is the space $V$ itself. Also, for $1 \leq i \leq n$, let $\LLL_i$ be the line bundle over $X_{n,k,d}$ whose fiber over $\ell_\bullet = (\ell_1, \dots, \ell_n)$ is the line $\ell_i$.  

\begin{lemma}
    \label{lem:ring-generation}
    The ordinary cohomology ring $H^*(X_{n,k,d};\ZZ)$ is generated as a $\ZZ$-algebra by the Chern classes 
    $$p^*(c_1(\UUU_d)), p^*(c_2(\UUU_d)), \dots, p^*(c_d(\UUU_d)) \quad \text{and} \quad c_1(\LLL_1), c_1(\LLL_2), \dots, c_1(\LLL_n).$$
    The equivariant cohomology ring $H^*_T(X_{n,k,d};\ZZ)$ is generated as a $\ZZ[\ttt_k]$-algebra by the equivariant Chern classes 
    $$p^*(c_1^T(\UUU_d)), p^*(c_2^T(\UUU_d)), \dots, p^*(c_d^T(\UUU_d)) \quad \text{and} \quad c_1^T(\LLL_1), c_1^T(\LLL_2), \dots, c_1^T(\LLL_n).$$
\end{lemma}

\begin{proof}
 Since every element in the non-equivariant cohomology is obtained from an equivariant cohomology class by setting $t_i=0$ for all $i$ (cf. (\ref{eq:ordinary-equivariant-relationship})), it suffices to prove the equivariant case.

Let $\alpha \in H_T^*(X_{n,k,d};\ZZ)$. As in the proof of Lemma \ref{lem:leray-hirsch-decomposition}, the set  $\{f^*(\sigma_w) \,:\,  w\in \mathcal{W}_{n,d}\}$ forms a $\ZZ$-basis of $H^*(X_{n,d};\ZZ)$ where we set $U_d=\CC^d \subset \CC^k$. By the Leray-Hirsch Theorem, the image of $\alpha$ under the isomorphism is given by
\[
\sum_{w \in \mathcal{W}_{n,d}} a_w \otimes f^*(\sigma_w) \ \ \in H_T^*(\Gr(d,k);\ZZ)\otimes_{\ZZ} H^*(X_{n,d};\ZZ)
\]
for some $a_w \in H_T^*(\Gr(d,k);\ZZ)$. Therefore we have
\[
\alpha = \sum_{w \in \mathcal{W}_{n,d}} p^*(a_w) \cdot \sigma_w.
\]
The claim holds, since $p^*(a_w)$ is a polynomial in $p^*(c_1^T(\UUU_d)), p^*(c_2^T(\UUU_d)), \dots, p^*(c_d^T(\UUU_d))$ and $\sigma_w$ is a polynomial in $c_1^T(\LLL_1), c_1^T(\LLL_2), \dots, c_1^T(\LLL_n)$.
\end{proof}

    The cells $\{ C_w \,:\, w \in \WWW_{n,k,d} \}$ of $X_{n,k,d}$ factor as products of cells of $\Gr(d,\CC^k)$ and cells in $X_{n,d}$. We close this subsection by explaining this factorization, with the running example $$(n,k,d) = (8,6,3) \text{ and } w = [2,2,5,2,5,4,5,4] \in \WWW_{n,k,d}.$$ Elements of $C_w = U(w) \widehat{C}_w T/T$ have the form
    $$umT  =  \begin{bmatrix}
        1 & 0 & 0 & 0 & 0 & 0 \\
        0 & 1 & 0 & 0 & 0 & 0 \\
        0 & a_{3,2} & 1 & 0 & 0 & 0 \\
        0 & a_{4,2} & 0 & 1 & 0 & 0 \\
        0 & a_{5,2} & 0 & a_{5,4} & 1 & 0 \\
        0 & a_{6,2} & 0 & a_{6,4} & a_{6,5} & 1
    \end{bmatrix} \cdot \begin{bmatrix}
        0 & 0 & 0 & 0 & 0 & 0 & 0 & 0 \\
        1 & 1 & \star & 1 & \star & \star & \star & \star \\
        0 & 0 & 0 & 0 & 0 & 0 & 0 & 0 \\
        0 & 0 & 0 & 0 & 0 & 1 & 0 & 1 \\
        0 & 0 & 1 & 0 & 1 & 0 & 1 & \star \\
        0 & 0 & 0 & 0 & 0 & 0 & 0 & 0 
    \end{bmatrix} \cdot T $$
    where the $a_{i,j}$ and $\star$'s are uniquely determined complex numbers.

    Write $\PPP_d(k-d)$ for the set of partitions $\lambda$ with $\leq d$ parts such that $\lambda_1 \leq k-d$. There is a  bijection between $\PPP_d(k-d)$ and the family ${[k] \choose d}$ of $d$-element subsets of $[k]$. If $$\lambda = (\lambda_1 \geq \lambda_2 \geq \cdots \geq \lambda_d) \in \PPP_d(k-d),$$we write $$S(\lambda) := \{ \lambda_1 + d - 1, \lambda_2 + d - 2, \dots, \lambda_d \}$$ for the corresponding subset.  If $\lambda \in \PPP_d(k-d)$, write $\mathring{X}_\lambda \subset \Gr(d,\CC^k)$ for the {\em open Schubert cell}  consisting of subspaces represented by matrices whose pivots are in the columns indexed by $S(\lambda)$.
        
     The \emph{support} of $w \in \WWW_{n,k,d}$, is $\supp(w) := \{i : w_j = i \text{ for some } j \} \subset [k]$. Let $\lambda \in \PPP_d(k-d)$ be such that $S(\lambda) = \supp(w)$. We have the following subgroups of $U(w)$:
    \begin{equation}
        U_{\gr}(w) := \{ u \in U(w): u_{i,j} = 0 \text{ if } ( j \not\in S(\lambda) \text{ or } i \in S(\lambda) ) \text{ and } i \neq j \},
    \end{equation}
    \begin{equation}
        U_{\ell}(w) := \{ u \in U(w): u_{i,j} = 0 \text{ if } ( j \not\in S(\lambda) \text{ or } i \not\in S(\lambda) ) \text{ and } i \neq j \}.
    \end{equation}
    Note that $U_{\gr}(w), U_\ell(w)$ depend only on $\supp(w)$. In our example $w = [2,2,5,2,5,4,5,4]$ we have $\supp(w) = \{2,4,5\}$ and
    $$U_\gr(w) = \left\{ \begin{bmatrix} 1 & 0 & 0 & 0 & 0 & 0 \\
    0 & 1 & 0 & 0 & 0 & 0 \\
    0 & b_{3,2} & 1 & 0 & 0 & 0 \\
    0 & 0 & 0 & 1 & 0 & 0 \\
    0 & 0 & 0 & 0 & 1 & 0 \\
    0 & b_{6,2} & 0 & b_{6,4} & b_{6,5} & 1\end{bmatrix} \right\}, \quad 
    U_\ell(w) = \left\{
    \begin{bmatrix}
        1 & 0 & 0 & 0 & 0 & 0 \\
        0 & 1 & 0 & 0 & 0 & 0 \\
        0 & 0 & 1 & 0 & 0 & 0 \\
        0 & b_{4,2} & 0 & 1 & 0 & 0 \\
        0 & b_{5,2} & 0 & b_{5,4} & 1 & 0 \\
        0 & 0 & 0 & 0 & 0 & 1 
    \end{bmatrix}
    \right\}$$
    where the $b_{i,j}$ are complex numbers. For any $w \in \WWW_{n,k,d}$, it can be shown using column operations that
    \begin{equation}
        U(w) = U_\gr(w) \cdot U_\ell(w)
    \end{equation}
    as matrix groups. Furthermore, the multiplication map $U_\gr(w) \times U_\ell(w) \to U(w)$ is bijective. If $umT \in C_w$, there exist unique $u_\gr \in U_\gr(w), u_\ell \in U_\ell(w)$ such that
    \begin{equation}
        umT = u_\gr u_\ell m T.
    \end{equation}

    If $S(\lambda) = \supp(w) = \{i_1 < \cdots < i_d\}$, let $m_\lambda$ be the $k \times d$ matrix with 1's at positions $(i_j,j)$ and 0's elsewhere. In our example we have 
    $$m_\lambda = \begin{bmatrix}
        0 & 0 & 0 \\
        1 & 0 & 0 \\
        0 & 0 & 0 \\
        0 & 1 & 0 \\
        0 & 0 & 1 \\
        0 & 0 & 0
    \end{bmatrix}.$$
    Given $u m T = u_\gr u_\ell m T \in C_w$, we may factor $m$ uniquely as
    $m = m_\lambda \cdot m'$ where $m' \in \MMM_{n,d}$. In our example, this factorization has the form
$$   
        m = \begin{bmatrix}
        0 & 0 & 0 & 0 & 0 & 0 & 0 & 0 \\
        1 & 1 & \star & 1 & \star & \star & \star & \star \\
        0 & 0 & 0 & 0 & 0 & 0 & 0 & 0 \\
        0 & 0 & 0 & 0 & 0 & 1 & 0 & 1 \\
        0 & 0 & 1 & 0 & 1 & 0 & 1 & \star \\
        0 & 0 & 0 & 0 & 0 & 0 & 0 & 0 
    \end{bmatrix} = 
    \begin{bmatrix}
        0 & 0 & 0 \\
        1 & 0 & 0 \\
        0 & 0 & 0 \\
        0 & 1 & 0 \\
        0 & 0 & 1 \\
        0 & 0 & 0
    \end{bmatrix} \cdot 
    \begin{bmatrix}
        1 & 1 & \star & 1 & \star & \star & \star & \star \\
        0 & 0 & 0 & 0 & 0 & 1 & 0 & 1 \\
        0 & 0 & 1 & 0 & 1 & 0 & 1 & \star 
    \end{bmatrix} = m_\lambda \cdot m'.
$$  
Let $w' \in \WWW_{n,d}$ be the word obtained from $w$ by replacing the smallest letter with 1, the next smallest letter with 2, and so on. In our running example we have 
$$w = [2,2,5,2,5,4,5,4] \, \Rightarrow \, w' = [1,1,3,1,3,2,3,2].$$ The matrix $m'$ in the factorization $m = m_\lambda \cdot m'$ fits the pattern of $w'$.

On the other hand, the product $u_\ell m_\lambda$ may be expressed as $u_\ell m_\lambda = m_\lambda u_\ell'$ for a unique $u_\ell' \in U_d$ obtained by `compressing' the matrix $u_\ell$. In our example this has the form
$$
u_\ell \cdot m_\lambda = \begin{bmatrix}
        1 & 0 & 0 & 0 & 0 & 0 \\
        0 & 1 & 0 & 0 & 0 & 0 \\
        0 & 0 & 1 & 0 & 0 & 0 \\
        0 & b_{4,2} & 0 & 1 & 0 & 0 \\
        0 & b_{5,2} & 0 & b_{5,4} & 1 & 0 \\
        0 & 0 & 0 & 0 & 0 & 1 
    \end{bmatrix} \cdot 
    \begin{bmatrix}
        0 & 0 & 0 \\
        1 & 0 & 0 \\
        0 & 0 & 0 \\
        0 & 1 & 0 \\
        0 & 0 & 1 \\
        0 & 0 & 0
    \end{bmatrix} = 
    \begin{bmatrix}
        0 & 0 & 0 \\
        1 & 0 & 0 \\
        0 & 0 & 0 \\
        0 & 1 & 0 \\
        0 & 0 & 1 \\
        0 & 0 & 0
    \end{bmatrix} \cdot
    \begin{bmatrix}
        1 & 0 & 0 \\
        b_{4,2} & 1 & 0 \\
        b_{5,2} & b_{5,4} & 1 
    \end{bmatrix} = m_\lambda \cdot u'_\ell.
$$
   Putting these matrix identities together, for $um T \in C_w$, we have the equality
\begin{equation}
\label{eq:product-factorization}
    u \cdot m = u_\gr \cdot u_\ell \cdot m_\lambda \cdot m' = (u_\gr \cdot m_\lambda) \cdot (u_\ell' \cdot m').
\end{equation}
The product $u_\gr \cdot m_\lambda$ represents a point in $\mathring{X}_\lambda$ and the product $u_\ell' \cdot m'$ represents a point in $C_w'$. The identity \eqref{eq:product-factorization} therefore leads to an isomorphism
\begin{equation}
\label{eq:cell-product}
    C_w \cong \mathring{X}_\lambda \times C_{w'}.
\end{equation}
Although we have no further use for the isomorphism \eqref{eq:cell-product} in this paper, we are hopeful that it will assist in the study of $X_{n,k,d}$, for example in finding cohomology representatives for the closures of the $C_w$.

\subsection{Chern class relations and rational cohomology} Lemma~\ref{lem:ring-generation} gives a generating set for the equivariant cohomology of $X_{n,k,d}$. In this subsection we use maps between vector bundles to give relations between these generators. If $\EEE \to X$ is a $T$-equivariant vector bundle of rank $r$ and $f$ is a symmetric polynomial in $r$ variables, we write $f(\EEE) \in H^*_T(X_{n,k,d};\ZZ)$ for the evaluation of $f$ at the Chern roots of $\EEE$. 

\begin{lemma}
    \label{lem:chern-class-relations}
    For $1 \leq i \leq n$, let $x_i := c_1^T(\LLL_i) \in H^2_T(X_{n,k,d};\ZZ)$. In the equivariant cohomology ring $H^*_T(X_{n,k,d};\ZZ)$ we have the relations
    \begin{itemize}
        \item $e_r(\ttt_k) - e_{r-1}(\ttt_k) h_1(p^*(\UUU_d)) + \cdots + (-1)^r h_r(p^*(\UUU_d)) = 0$ for all $r > k-d$,
        \item $e_r(\xxx_n) - e_{r-1}(\xxx_n) h_1(p^*(\UUU_d)) + \cdots + (-1)^r h_r(p^*(\UUU_d)) = 0$ for all $r > n-d$, and
        \item $x_i^d - x_i^{d-1} e_1(p^*(\UUU_d)) + \cdots + (-1)^r e_d(p^*(\UUU_d)) = 0$ for $i = 1, \dots, n$.
    \end{itemize}
\end{lemma}

\begin{proof}
 Write $\CC^k$ for the rank $k$ bundle over $X_{n,k,d}$ whose fiber over any point is the vector space $\CC^k$ with the natural action of $T = (\CC^*)^k$.
    For the relation in the first bullet point, form the tautological quotient bundle $\CC^k/\UUU_d$ over $\Gr(d,\CC^k)$ of rank $k-d$. We have a short exact sequence of vector bundles
    \begin{equation}
        0 \to \UUU_d \to \CC^k \to \CC^k/\UUU_d \to 0
    \end{equation}
    over $\Gr(d,\CC^k)$ with the relation 
    \begin{equation}
        c^T(\UUU_d) \cdot c^T(\CC^k/\UUU_d) = c^T(\CC^k) = (1 + t_1) \cdots + (1 + t_k).
    \end{equation}
    Rearranging, we get the relation
    \begin{equation}
    \label{eqn:equivariant-rational-expression-one}
        c^T(\CC^k/\UUU_d) = \frac{(1 + t_1) \cdots (1 + t_k)}{c^T(\UUU_d)} = \frac{(1 + t_1) \cdots (1 + t_k)}{(1 + y_1) \cdots (1 + y_d)}
    \end{equation}
    where $y_1, \dots, y_d$ are the equivariant Chern roots of $\UUU_d$; we have $e_i(\yyy_d) = c_i^T(\UUU_d)$. Since $\CC^k/\UUU_d$ has rank $k-d$, we have $c_r^T(\CC^k/\UUU_d) = 0$ for $r > k-d$ and the relation in the first bullet point arises from taking the degree $r$ component in \eqref{eqn:equivariant-rational-expression-one} and pulling back under the map $p^*$.

    For the second bullet point, form the direct sum $\LLL_1 \oplus \cdots \oplus \LLL_n$; this is a vector bundle over $X_{n,k,d}$ of rank $n$. By the definition of $X_{n,k,d}$, we have a surjective map of bundles $\LLL_1 \oplus \cdots \oplus \LLL_n \twoheadrightarrow p^*(\UUU_d)$ given by vector addition in each fiber. Let $\KKK_{n-d}$ be the kernel of this map; it is a vector bundle over $X_{n,k,d}$ of rank $n-d$. We have a short exact sequence
    \begin{equation}
        0 \to \KKK_{n-d} \to \LLL_1 \oplus \cdots \oplus \LLL_n \to p^*(\UUU_d) \to 0
    \end{equation}
    and the total Chern class
    \begin{equation}
        c^T(\LLL_1 \oplus \cdots \oplus \LLL_n) = c^T(\LLL_1) \cdot \cdots \cdot c^T(\LLL_n) = (1 + x_1) \cdots (1 + x_n).
    \end{equation}
    The relation in the second bullet point follows from the fact that $c_r^T(\KKK_{n-d}) = 0$ for $r > n-d$.

    For the third bullet point, let $1 \leq i \leq n$. The line bundle $\LLL_i$ embeds into the vector bundle $p^*(\UUU_d)$. The quotient $p^*(\UUU_d)/\LLL_i$ has rank $d-1$. We have a short exact sequence
    \begin{equation}
        0 \to \LLL_i \to p^*(\UUU_d) \to p^*(\UUU_d)/\LLL_i \to 0
    \end{equation}
    of bundles over $X_{n,k,d}$. The relation in the third bullet point follows from $c_d^T(p^*(\UUU_d)/\LLL_i) = 0$.
\end{proof}

The relations in Lemma~\ref{lem:chern-class-relations} coincide with those in Definition~\ref{def:j-t-ideal}. Theorem~\ref{thm:borel-presentation} says that these are all the relations we need to present $H_T^*(X_{n,k,d};\ZZ)$. There will be some added complexity in proving this holds with integer coefficients, but we can present $H_T^*(X_{n,k,d};\QQ)$ right away. 

\begin{proposition}
    \label{prop:rational-cohomology-presentation}
    Let $n \geq k \geq d$ be positive integers. We have the presentation
    \begin{equation}
        H^*_T(X_{n,k,d};\QQ) = \QQ[\xxx_n,\yyy_d,\ttt_k]^{\symm_d}/(J_{n,k,d}^{\QQ,\ttt} \cap \QQ[\xxx_n,\yyy_d,\ttt_k]^{\symm_d}).
    \end{equation}
\end{proposition}

\begin{proof}
    It follows from Lemma~\ref{lem:ring-generation} that $H^*_T(X_{n,k,d};\QQ)$ is concentrated in even degrees. The Universal Coefficient Theorem and Lemma~\ref{lem:cohomology-ring-ranks} imply that $H^*_T(X_{n,k,d};\QQ)$ is a free $\QQ[\ttt_k]$-module of rank $\frac{k!}{(k-d)!} \cdot \Stir(n,d)$. 

    Lemma~\ref{lem:chern-class-relations} implies that we have a well-defined $\QQ[\ttt_k]$-algebra homomorphism
    \begin{equation}
        \varphi: \QQ[\xxx_n,\yyy_d,\ttt_k]^{\symm_d}/(J_{n,k,d}^{\QQ,\ttt} \cap \QQ[\xxx_n,\yyy_d,\ttt_k]^{\symm_d}) \to H^*_T(X_{n,k,d};\QQ)
    \end{equation}
    which sends $x_i$ to $c_1^T(\LLL_i)$ and the elementary symmetric polynoamial $e_i(\yyy_d)$ to $c_i^T(p^*(\UUU_d))$. Lemma~\ref{lem:ring-generation} implies that $\varphi$ is a surjection. 
    
    We know that the domain of $\varphi$ is a free $\QQ[\ttt_k]$-module of rank $\frac{k!}{(k-d)!} \cdot \Stir(n,d)$.  Lemma~\ref{lem:t-invariant-free-module} implies that the target of $\varphi$ is a free $\QQ[\ttt_k]$-module of the same rank. Lemma~\ref{lem:surjection-lemma} applies to show that the surjection $\varphi$ is in fact an isomorphism.
\end{proof}

\subsection{Integral cohomology presentation} The aim of this subsection is to prove Theorem~\ref{thm:borel-presentation}. The method of proof in Proposition~\ref{prop:rational-cohomology-presentation} will apply if we can show that $\ZZ[\xxx_n,\ccc_d,\ttt_k]/I_{n,k,d}$ is a free $\ZZ[\ttt_k]$-module of rank $\frac{k!}{(k-d)!} \cdot \Stir(n,d)$. We start by considering an ideal in $\ZZ[\xxx_n,\yyy_d]$.

\begin{defn}
    \label{def:j-z-ideal}
    Let $J_{n,k,d} \subseteq \ZZ[\xxx_n,\yyy_d]$ be the ideal with the same generating set as 
    in Definition~\ref{def:j-ideal-definition}
\end{defn}

Explicitly, the ideal $J_{n,k,d} \subseteq \ZZ[\xxx_n,\yyy_d]$ is generated by the polynomials
\begin{itemize}
    \item $h_r(\yyy_d)$ for all $r > k-d$,
    \item $e_r(\xxx_n) - e_{r-1}(\xxx_n) h_1(\yyy_d) + \cdots + (-1)^r h_r(\yyy_d)$ for all $r > n-d$, and
    \item $x_i^d - x_i^{d-1} e_1(\yyy_d) + \cdots + (-1)^d e_d(\yyy_d)$ for $i = 1, \dots, n$.
\end{itemize}
To shorten our statements and arguments, we introduce a notational shorthand for these generators.

\noindent
\begin{quote}
{\bf Notation.} {\em For the remainder of this subsection, let $G \subseteq \ZZ[\xxx_n,\yyy_d]$ denote the set of generators in the above three bullet points.}
\end{quote}

The elements of $G$ lie in the $\symm_d$-invariant subring $\ZZ[\xxx_n,\yyy_d]^{\symm_d} \subseteq \ZZ[\xxx_n,\yyy_d]$.
If we were working over $\QQ$, we could use the Reynolds operator $(-)^\natural = \frac{1}{d!} \sum_{w \in \symm_d} w \cdot (-)$ to show that $G$ generates $J_{n,k,d} \cap \ZZ[\xxx_n,\yyy_d]^{\symm_d}$. Although we cannot divide by $d!$ in our setting, following is an approximation of this fact over $\ZZ$.

\begin{lemma}
    \label{lem:integer-reynolds}
    Let $f \in J_{n,k,d} \cap \ZZ[\xxx_n,\yyy_d]^{\symm_d}$. Then $d! \cdot f \in \ZZ[\xxx_n,\yyy_d]^{\symm_d} \cdot G$, where $\ZZ[\xxx_n,\yyy_d]^{\symm_d} \cdot G$ is the ideal in $\ZZ[\xxx_n,\yyy_d]^{\symm_d}$ generated by $G$.
\end{lemma}

\begin{proof}
    There exist polynomials $h_g \in \ZZ[\xxx_n,\yyy_d]$ so that 
    $$f = \sum_{g \, \in \, G} h_g \cdot g.$$
    Applying $\sum_{w  \in  \symm_d} w \cdot (-)$ to both sides and using the fact that both $f$ and the elements of $G$ are $\symm_d$-invariant, we have
    $$d! \cdot f = \sum_{g \, \in \, G} \left( \sum_{w \, \in \, \symm_d} w \cdot h_g \right) \cdot g$$
    which is an element of $\ZZ[\xxx_n,\yyy_d]^{\symm_d} \cdot G.$
\end{proof}

Lemma~\ref{lem:integer-reynolds} implies that 
\begin{equation}
    \label{eq:containment-up-to-factorial}
    d! \cdot (J_{n,k,d} \cap \ZZ[\xxx_n,\yyy_d]^{\symm_d}) \subseteq \ZZ[\xxx_n,\yyy_d]^{\symm_d} \cdot G.
\end{equation}
We certainly have $\ZZ[\xxx_n,\yyy_d]^{\symm_d} \cdot G \subseteq J_{n,k,d}.$
We want to prove 
\begin{equation}
J_{n,k,d} \cap \ZZ[\xxx_n,\yyy_d]^{\symm_d} = \ZZ[\xxx_n,\yyy_d]^{\symm_d} \cdot G.
\end{equation}
Thanks to the containment \eqref{eq:containment-up-to-factorial}, it is enough to show that $\ZZ[\xxx_n,\yyy_d]^{\symm_d}/(\ZZ[\xxx_n,\yyy_d]^{\symm_d} \cdot G)$ is a free $\ZZ$-module. We first describe a $\ZZ$-spanning set for this quotient.

\begin{lemma}
\label{lem:Z-spanning-set}
The quotient ring $\ZZ[\xxx_n,\yyy_d]^{\symm_d}/(\ZZ[\xxx_n,\yyy_d]^{\symm_d} \cdot G)$ is spanned over $\ZZ$ by the set
$$ \CCC_{n,k,d} := \{ x_1^{a_1} \cdots x_n^{a_n} \,:\, (a_1, \dots, a_n) \text{ is $(n,d)$-substaircase} \} \times \{s_\lambda(\yyy_d) \,:\, \lambda \subseteq (k-d)^d \}.$$
\end{lemma}

\begin{proof}
    The larger ring $\ZZ[\xxx_n,\yyy_d]^{\symm_d}$ has a $\ZZ$-basis given by the set of products of the form
    $$ \hat{\CCC}_{n,k,d} = \{ x_1^{a_1} \cdots x_n^{a_n} \cdot s_\lambda(\yyy_d) \,:\, a_i \geq 0 \text{ and } \ell(\lambda) \leq d \},$$
    so this larger set $\hat{\CCC}_{n,k,d}$ descends to a $\ZZ$-spanning set of the quotient $\ZZ[\xxx_n,\yyy_d]^{\symm_d}/(\ZZ[\xxx_n,\yyy_d]^{\symm_d} \cdot G)$. Suppose we have an element $x_1^{a_1} \cdots x_n^{a_n} \cdot s_\lambda(\yyy_d) \in \hat{\CCC}_{n,k,d} - \CCC_{n,k,d}$. We aim to show 
    $$ x_1^{a_1} \cdots x_n^{a_n} \cdot s_\lambda(\yyy_d) \in \mathrm{span}_\ZZ(\CCC_{n,k,d}) \mod \ZZ[\xxx_n,\yyy_d]^{\symm_d} \cdot G.$$
    We induct on the sum $a := a_1 + \cdots + a_n$ of the exponents of the $x$-variables.

    If $a = 0$ then $\lambda_1 > k-d$. Since $h_r(\yyy_d) \in G$ for all $r > k-d$, expansion along the first row of the Jacobi-Trudi identity
    \begin{equation}
        s_\lambda(\yyy_d) = \det \left(  h_{\lambda_i - i + j}(\yyy_d) \right)_{1 \leq i,j \leq \ell(\lambda)}
    \end{equation}
    implies that $s_\lambda(\yyy_d) \in \ZZ[\yyy_d]^{\symm_d} \cdot G \subseteq \ZZ[\xxx_n,\yyy_d]^{\symm_d} \cdot G$. This concludes the case $a = 0$.

    Now assume $a > 0$. If $\lambda_1 > k-d$ we have $x_1^{a_1} \cdots x_n^{a_n} \cdot s_\lambda(\yyy_d) \in  \ZZ[\xxx_n,\yyy_d]^{\symm_d} \cdot G$ by the reasoning of the last paragraph. So assume $\lambda_1 \leq k-d$. The sequence $(a_1, \dots, a_n)$ is not $(n,d)$-substaircase.  For any polynomial $g \in G$, let $\bar{g} \in \ZZ[\xxx_n]$ be the polynomial obtained by setting the $y$-variables in $g$ equal to 0. Lemmas~\ref{lem:substaircase-standard-monomial} and \ref{lem:integral-substaircase} yield polynomials $h_g \in \ZZ[\xxx_n]$ such that
    \begin{equation}
        \initial_\prec \left( \sum_{g \, \in \, G} h_g \cdot \bar{g}  \right) = x_1^{a_1} \cdots x_n^{a_n}
    \end{equation}
    where $\prec$ is the lexicographical term order with respect to $x_1 > \cdots > x_n$. Multiplying through by $s_\lambda(\yyy_d)$ and replacing each $\bar{g}$ with the corresponding generator $g \in G$, we see that 
    \begin{equation}
    \label{eqn:rewriting-Z-two}
    x_1^{a_1} \cdots x_n^{a_n} \cdot s_\lambda(\yyy_d) 
 \equiv \\ \begin{array}{c}\text{a $\ZZ$-linear combination of terms of the form} \\ \text{$x_1^{b_1} \cdots x_n^{b_n} \cdot s_\mu(\yyy_d)$ such that $b < a$ or} \\ \text{$b=a$, $\lambda = \mu$, and $(b_1,\dots,b_n) \prec (a_1,\dots,a_n)$}\end{array} \mod \ZZ[\xxx_n,\yyy_d]^{\symm_d} \cdot G
    \end{equation}
    where $b := b_1 + \cdots + b_n$. By induction on $a$ and lexicographical order, every term in the $\ZZ$-linear combination on the right hand side lies in the $\ZZ$-span of $\CCC_{n,k,d}$ modulo $\ZZ[\xxx_n,\yyy_d]^{\symm_d} \cdot G$.
\end{proof}

We will show that the $\ZZ$-spanning set $\CCC_{n,k,d}$ in Lemma~\ref{lem:Z-spanning-set} is in fact a $\ZZ$-basis. This will also show that $J_{n,k,d} \cap \ZZ[\xxx_n,\yyy_d]^{\symm_d} = G \cdot \ZZ[\xxx_n,\yyy_d]^{\symm_d}$. The proof  uses  makes use of the  standard presentation of the cohomology of the Grassmannian:
\begin{equation}
\label{eqn:grassmannian-cohomology-presentation-1}
    H^*(\Gr(d,\CC^k);\ZZ) = \ZZ[\yyy_d]^{\symm_d}/(s_\lambda(\yyy_d) \,:\, \lambda \not\subseteq (k-d)^d).
\end{equation}
The defining ideal in \eqref{eqn:grassmannian-cohomology-presentation-1} is generated by Schur polynomials $s_\lambda(\yyy_d)$ indexed by partitions $\lambda$ whose Young diagrams do not fit inside a $d$-by-$(k-d)$ box. The variables $y_1, \dots, y_d$ represent the Chern roots of the tautological vector bundle $\UUU_d$ over $\Gr(d,\CC^k)$. The family 
$$\{ s_\lambda(\yyy_d) \,:\, \lambda \subseteq (k-d)^d \}$$
of Schur polynomials corresponding to partitions which fit inside a $d$-by-$(k-d)$ box descends to a $\ZZ$-basis of $H^*(\Gr(d,\CC^k);\ZZ)$. See \cite[Sec. 9.4]{Fulton} for a textbook treatment of these facts.

\begin{lemma}
    \label{lem:J-generating-set-over-Z}
    The set $G$ generates the ideal $J_{n,k,d} \cap \ZZ[\xxx_n,\yyy_d]^{\symm_d}$ as an ideal in $\ZZ[\xxx_n,\yyy_d]^{\symm_d}$. Furthermore, the set $\CCC_{n,k,d}$ of Lemma~\ref{lem:Z-spanning-set} descends to a $\ZZ$-basis of
    $$\ZZ[\xxx_n,\yyy_d]^{\symm_d}/(J_{n,k,d} \cap \ZZ[\xxx_n,\yyy_d]^{\symm_d}) = \ZZ[\xxx_n,\yyy_d]^{\symm_d}/(\ZZ[\xxx_n,\yyy_d]^{\symm_d} \cdot G).$$
\end{lemma}

\begin{proof}
    As discussed before Lemma~\ref{lem:Z-spanning-set},
    it is enough to show that the spanning set $\CCC_{n,k,d}$ in Lemma~\ref{lem:Z-spanning-set} is also linearly independent as a subset of $\ZZ[\xxx_n,\yyy_d]^{\symm_d}/(\ZZ[\xxx_n,\yyy_d]^{\symm_d} \cdot G)$. This implies both that $\ZZ[\xxx_n,\yyy_d]^{\symm_d}/(\ZZ[\xxx_n,\yyy_d]^{\symm_d} \cdot G)$ is a free $\ZZ$-module and also that $J_{n,k,d} \cap \ZZ[\xxx_n,\yyy_d]^{\symm_d} = \ZZ[\xxx_n,\yyy_d]^{\symm_d} \cdot G$. 
    
    Consider the ordinary cohomology ring $H^*(X_{n,k,d};\ZZ)$ of $X_{n,k,d}$.
    We have a $\ZZ$-algebra homomorphism
    \begin{equation}
    \psi: \ZZ[\xxx_n,\yyy_d]^{\symm_d}/(\ZZ[\xxx_n,\yyy_d]^{\symm_d} \cdot G) \twoheadrightarrow H^*(X_{n,k,d};\ZZ)
    \end{equation}
    which sends $x_i \mapsto c_1(\LLL_i)$ and $e_i(\yyy_d) \mapsto c_i(p^*(\UUU_d))$. 
    The Leray-Hirsch Theorem gives a $\ZZ$-module isomorphism
    \begin{equation}
    H^*(X_{n,k,d};\ZZ) \cong H^*(X_{n,d};\ZZ) \otimes_\ZZ H^*(\Gr(d,\CC^k);\ZZ).
    \end{equation}
    Pawlowski and Rhoades proved \cite[Thm 5.12 (3)]{PR} that $H^*(X_{n,d};\ZZ)$ has presentation
    \begin{equation}
        H^*(X_{n,d};\ZZ) = \ZZ[\xxx_n]/(x_1^d, x_2^2, \dots, x_n^d, e_n(\xxx_n), e_{n-1}(\xxx_n), \dots, e_{n-d+1}(\xxx_n))
    \end{equation}
    where $x_i \leftrightarrow c_1(\LLL_i)$
    so and \cite[Proof of Thm. 5.12 (3)]{PR} that
    $$ \{ x_1^{a_1} \cdots x_n^{a_n} \,:\, (a_1, \dots, a_n) \text{ is $(n,d)$-substaircase} \}$$
    descends to a $\ZZ$-basis of $H^*(X_{n,d};\ZZ)$. Furthermore, the set $s_\lambda(\UUU_d)$ where for $\lambda \subseteq (k-d)^d$ is a $\ZZ$-basis of $H^*(\Gr(d,\CC^k);\ZZ)$.
    It follows that $\psi$ maps the spanning set $\CCC_{n,k,d}$ to a $\ZZ$-basis. In particular, the set $\CCC_{n,k,d}$ must be linearly independent over $\ZZ$.
\end{proof}

It may be interesting to find a purely algebraic proof of Lemma~\ref{lem:J-generating-set-over-Z}. Our next result shows that $\CCC_{n,k,d}$ is a $\ZZ[\ttt_k]$-basis of $\ZZ[\xxx_n,\yyy_d,\ttt_k]^{\symm_d}/I_{n,k,d}$.

\begin{lemma}
    \label{lem:Zt-basis}
    The set $\CCC_{n,k,d}$ of Lemma~\ref{lem:Z-spanning-set} descends to a $\ZZ[\ttt_k]$-module basis of $\ZZ[\xxx_n,\yyy_d,\ttt_k]^{\symm_d}/I_{n,k,d}$. In particular, the quotient $\ZZ[\xxx_n,\yyy_d,\ttt_k]^{\symm_d}/I_{n,k,d}$ is a free $\ZZ[\ttt_k]$-module of rank $\frac{k!}{(k-d)!} \cdot \Stir(n,d)$.
\end{lemma}

\begin{proof}
    Recall from Definition~\ref{def:t-augmentation} that
     $\CCC^\ttt_{n,k,d}$ is the set of products $c \cdot m_\ttt$ of elements $c \in \CCC_{n,k,d}$ with monomials $m_\ttt$ in the $t$-variables. We want to show that $\CCC^\ttt_{n,k,d}$ is a $\ZZ$-basis of $\ZZ[\xxx_n,\ccc_d,\ttt_k]/I_{n,k,d}$.

    We first show that $\CCC^\ttt_{n,k,d}$ spans $\ZZ[\xxx_n,\ccc_d,\ttt_k]/I_{n,k,d}$ over $\ZZ$. Let $m_\xxx \cdot f_\yyy \cdot m_\ttt \in \ZZ[\xxx_n,\yyy_d,\ttt_k]^{\symm_d}$, where $m_\xxx$ is a monomial in the $x$-variables, $f_\yyy \in \ZZ[\yyy_d]^{\symm_d}$ is a homogeneous symmetric polynomial in the $y$-variables, and $m_\ttt$ is a monomial in the $t$-variables. If $m_\xxx = 1$ and $f_\yyy \in \ZZ$, then $m_\xxx \cdot f_\yyy \cdot m_\ttt$ lies in $\mathrm{span}_\ZZ(\CCC^\ttt_{n,k,d})$. Otherwise, we argue as follows.
    
    By Lemma~\ref{lem:J-generating-set-over-Z}, there exist integers $\gamma_c \in \ZZ$ and polynomials $h_g \in \ZZ[\xxx_n,\ccc_d]$ such that 
    \begin{equation}
    \label{eqn:zt-basis-one}
    m_\xxx \cdot f_\yyy = \sum_{c \, \in \, \CCC_{n,k,d}} \gamma_c \cdot c + \sum_{g \, \in \, G} h_g \cdot g.
    \end{equation}
    Discarding redundant terms if necessary, we may assume that 
    \begin{itemize}
        \item $\gamma_c = 0$ unless $c$ has degree $\deg(m_\xxx \cdot f_\yyy)$, and
        \item $h_g$ is a homogeneous polynomial of degree $\deg(m_\xxx \cdot f_\yyy) - \deg(g)$.
    \end{itemize}
    We may multiply Equation~\eqref{eqn:zt-basis-one} by $m_\ttt$ to obtain
    \begin{equation}
    \label{eqn:zt-basis-two}
    m_\xxx \cdot f_\yyy \cdot m_\ttt = \sum_{c \, \in \, \CCC_{n,k,d}} \gamma_c \cdot c \cdot m_\ttt + \sum_{g \, \in \, G} h_g \cdot g \cdot m_\ttt.
    \end{equation}
    For every generator $g \in G$, there is a corresponding generator $\tilde{g} \in \ZZ[\xxx_n,\yyy_d,\ttt_k]^{\symm_d}$ of $I_{n,k,d}$ such that $g$ is obtained from $\tilde{g}$ by setting the $t$-variables to zero. 
    We have 
    \begin{equation}
    \label{eqn:zt-basis-three}
    m_\xxx \cdot f_\yyy \cdot m_\ttt = \sum_{c \, \in \, \CCC_{n,k,d}} \gamma_c \cdot c \cdot m_\ttt  + \sum_{g \in G} h_g \cdot (g-\tilde{g}) \cdot m_\ttt + \sum_{g \, \in \, G} h_g \cdot \tilde{g} \cdot m_\ttt.
    \end{equation}
    The right hand side of Equation~\eqref{eqn:zt-basis-three} may be analyzed as follows.
    \begin{itemize}
        \item The first sum $\sum_{c \in \CCC_{n,k,d}} \gamma_c \cdot c \cdot m_\ttt$ is a $\ZZ$-linear combination of elements of $\CCC^\ttt_{n,k,d}$.
        \item The second sum $ \sum_{g \in G} h_g \cdot (g-\tilde{g}) \cdot m_\ttt$ is a $\ZZ$-linear combination of polynomials of the form $m'_\xxx \cdot f'_\yyy \cdot m'_\ttt$ such that
        $$\deg(m'_\xxx \cdot f'_\yyy \cdot m'_\ttt) = \deg(m_\xxx \cdot f_\yyy \cdot m_\ttt)$$
        but
        $$\deg(m'_\xxx \cdot f'_\yyy) < \deg(m_\xxx \cdot f_\yyy) \quad \text{and} \quad \quad \deg(m'_\ttt) > \deg(m_\ttt).$$
        \item The third sum $\sum_{g \in G} h_g \cdot \tilde{g} \cdot m_\ttt$ is an element of $I_{n,k,d}$.
    \end{itemize}
    By induction on $\deg(m_\xxx \cdot f_\yyy)$, we conclude that $m_\xxx \cdot f_\yyy \cdot m_\ttt$ lies in the $\ZZ$-span of $\CCC^\ttt_{n,k,d}$ modulo $I_{n,k,d}$.

    Any $\ZZ$-linear dependence of $\CCC^\ttt_{n,k,d}$ modulo $I_{n,k,d}$ would induce a $\QQ$-linear independence of $\CCC^\ttt_{n,k,d}$ modulo $\QQ \otimes_\ZZ I_{n,k,d}$ inside the ring $\QQ[\xxx_n,\yyy_d,\ttt_k]^{\symm_d}$.
    By Lemma~\ref{lem:t-invariant-free-module} such a $\QQ$-linear dependence cannot occur, provided $\CCC_{n,k,d}$ descends to a $\QQ$-basis of $\QQ[\xxx_n,\yyy_d]^{\symm_d}/(J_{n,k,d}^\QQ \cap \QQ[\xxx_n,\yyy_d]^{\symm_d})$. This latter fact is true by Lemma~\ref{lem:J-generating-set-over-Z}.
\end{proof}

We have all the tools we need to present $H^*_T(X_{n,k,d};\ZZ)$. We recall the statement of Theorem~\ref{thm:borel-presentation}.

\noindent
{\bf Theorem 1.2.} {\em For positive integers $n \geq k \geq d$, let $I_{n,k,d} \subseteq \ZZ[\xxx_n,\yyy_d,\ttt_k]^{\symm_d}$ be the ideal generated by 
\begin{itemize}
    \item $e_r(\ttt_k) - e_{r-1}(\ttt_k) h_1(\yyy_d) + \cdots + (-1)^r h_r(\yyy_d)$ for $r > k-d$, 
    \item $e_r(\xxx_n) - e_{r-1}(\xxx_n) h_1(\yyy_d) + \cdots + (-1)^r h_r(\yyy_d)$ for $r > n-d$, and
    \item $x_i^d - x_i^{d-1} e_1(\yyy_d) + \cdots + (-1)^d e_d(\yyy_d)$ for $i = 1, \dots, n$.
\end{itemize}
If $T = (\CC^*)^k$, the $T$-equivariant cohomology of $X_{n,k,d}$ has presentation
    \begin{equation}
        H^*_T(X_{n,k,d};\ZZ) = \ZZ[\xxx_n,\yyy_d,\ttt_k]^{\symm_d}/I_{n,k,d}
    \end{equation}
    where 
    \begin{itemize}
        \item $x_i$ represents the equivariant Chern class of the line bundle $\LLL_i$ over $X_{n,k,d}$ with fiber $\ell_i$ over $\ell_\bullet = (\ell_1, \dots, \ell_n)$,
        \item the variables $y_1, \dots, y_d$ represent the equivariant Chern roots of the pullback along $p: X_{n,k,d} \to \Gr(d,\CC^k)$ of the rank $d$ tautological vector bundle $\UUU_d$ over $\Gr(d,\CC^k)$, and
        \item $t_i \in H^2_T(X_{n,k,d};\ZZ)$ represents the equivariant Chern class of the tautological line bundle over the $i^{th}$ factor of $BT = (\PP^{\infty})^k$.
    \end{itemize}}

\begin{proof}
    Lemma~\ref{lem:chern-class-relations} implies that we have a well-defined $\ZZ[\ttt_k]$-homomorphism $$\varphi: \ZZ[\xxx_n,\yyy_d,\ttt_k]^{\symm_d}/I_{n,k,d} \to H^*_T(X_{n,k,d};\ZZ)$$which maps the variables to the indicated Chern classes. Lemma~\ref{lem:cohomology-ring-ranks} says that the target of $\varphi$ is a free $\ZZ[\ttt_k]$-module of rank $\frac{k!}{(k-d)!} \cdot \Stir(n,d)$. Lemma~\ref{lem:Zt-basis} implies that the domain of $\varphi$ is a free $\ZZ[\ttt_k]$-module of the same rank. Lemma~\ref{lem:ring-generation} says that $\varphi$ is surjective. Finally, Lemma~\ref{lem:surjection-lemma} implies that $\varphi$ is an isomorphism.
\end{proof}

Recall that $X_{n,k} = X_{n,k,k}$ is the moduli space of $n$-tuples of lines $(\ell_1, \dots, \ell_n)$ in $\CC^k$ which satisfy $\ell_1 + \cdots + \ell_n = \CC^k$.  Pawlowski and Rhoades posed \cite[Prob. 9.8]{PR} the problem of computing the torus equivaiant cohomology of $X_{n,k}$. The solution is as follows. Let $I_{n,k} \subseteq \ZZ[\xxx_n,\ttt_k]$ be the ideal generated by
\begin{itemize}
    \item $x_i^k - x_i^{k-1} e_1(\ttt_k) + \cdots + (-1)^k e_k(\ttt_k)$ for $1 \leq i \leq n$, and
    \item $e_r(\xxx_n) - e_{r-1}(\xxx_n) h_1(\ttt_k) + \cdots + (-1)^r h_r(\ttt_k)$ for $r > n-k$.
\end{itemize}

\begin{corollary}
    \label{cor:pr-problem}
    For positive integers $n \geq k$, the $T$-equivariant cohomology of $X_{n,k}$ has presentation
    \begin{equation}
        H^*_T(X_{n,k};\ZZ) = \ZZ[\xxx_n,\ttt_k]/I_{n,k}
    \end{equation}
    where $x_i$ and $t_i$ represent the equivariant Chern classes of the line bundles of Theorem~\ref{thm:borel-presentation}.
\end{corollary}

\begin{proof}
    In the notation of Theorem~\ref{thm:borel-presentation}, since $d = k$ thep pullback under $p: X_{n,k} \to \Gr(k,\CC^k)$ of the tautological bundle $\UUU_k$ over (the one-point space) $\Gr(k,\CC^k)$ is the bundle $\CC^k$ with its natural $T$-action. It follows that $e_r(\yyy_k) = e_r(\ttt_k)$ in as elements of $H^*_T(X_{n,k};\ZZ)$ for all $r$. This yields a $\ZZ[\ttt_k]$-algebra map
    $$\varphi: \ZZ[\xxx_n,\ttt_k]/I_{n,k} \to H^*_T(X_{n,k};\ZZ)$$
    which sends variables to Chern classes. Lemma~\ref{lem:ring-generation} together with the bundle equality $p^*(\UUU_k) = \CC^k$ implies that $\varphi$ is a surjection. 

    In order to show that $\varphi$ is an isomorphism, it is enough (by Theorem~\ref{thm:borel-presentation} and Lemma~\ref{lem:surjection-lemma}) to prove $\ZZ[\xxx_n,\yyy_k,\ttt_k]^{\symm_k}/I_{n,k,k} \cong \ZZ[\xxx_n,\ttt_k]/I_{n,k}$ as $\ZZ[\ttt_k]$-algebras where $\symm_d$ acts on the $y$-variables. Indeed, the ideal membership 
    \begin{center} $\sum_{a+b = r} (-1)^b e_a(\ttt_k) h_b(\yyy_k) \in I_{n,k,k}$ for $r > 0$ \end{center} implies that $f(\ttt_k) - f(\yyy_k) \in I_{n,k,k}$ for any symmetric polynomial $f$. This given, the surjection $\ZZ[\xxx_n,\yyy_k,\ttt_k]^{\symm_k} \twoheadrightarrow \ZZ[\xxx_n,\ttt_k]$ sending $e_r(\yyy_d) \mapsto e_r(\ttt_k)$ and the inclusion $\ZZ[\xxx_n,\ttt_k] \hookrightarrow \ZZ[\xxx_n,\yyy_k,\ttt_k]^{\symm_k}$ are easily seen to induce mutually inverse $\ZZ[\ttt_k]$-module isomorphisms between $\ZZ[\xxx_n,\yyy_k,\ttt_k]^{\symm_k}/I_{n,k,k}$ and $\ZZ[\xxx_n,\ttt_k]/I_{n,k}$.
\end{proof}

\section{Towards a GKM Presentation}
\label{sec:GKM}

Let $T = (\CC^*)^k$ act on a manifold $X$ and let $X^T$ be the set of $T$-fixed points with inclusion map $i: X^T \hookrightarrow X$. When  $X^T$  is finite, Goresky, Kottwitz, and MacPherson introduced \cite{GKM} described conditions on $X$ (see \cite{GZ}) which
\begin{itemize}
    \item guarantee that the restriction map $i^*: H^*_T(X;\QQ) \rightarrow H^*_T(X^T;\QQ)$ is injective, and
    \item combinatorially describe the image of $i^*$ in $H^*_T(X^T;\QQ) \cong \bigoplus_{x \in X^T} \QQ[\ttt_k]$ in terms of the $1$-dimensional $T$-orbits in $X$.
\end{itemize}
One of these {\em GKM conditions} on $X$ is compactness. The moduli space $X_{n,k,d}$ is not compact, and so does not directly fit into this framework. The purpose of this section is to show that, nevertheless, the restriction map $i^*: H^*_T(X;\QQ) \rightarrow H^*_T(X^T;\QQ)$ is an injection. The proof of this fact uses the orbit harmonics results of Section~\ref{sec:Quotient}.

Let $n \geq k \geq d$ be positive integers and consider the space $\Zpoints_{n,k,d} \subseteq \QQ^{n+d+k}$ of Definition~\ref{def:z-family}. The variety $\Zpoints_{n,k,d}$ breaks up into irreducible components as follows. 

For any injection $f: \{n+1, \dots, n+d\} \hookrightarrow \{n+d+1, \dots, n+d+k\}$ and any surjection $g: \{1, \dots, n\} \twoheadrightarrow \{n+1, \dots, n+d\}$, let $\ZZZ_{n,k,d}^{f,g} \subseteq \ZZZ_{n,k,d}$ be the locus of points $$(z_1, \dots, z_n; z_{n+1}, \dots, z_{n+d}; z_{n+d+1}, \dots, z_{n+d+k}) \in \Zpoints_{n,k,d}$$ such that 
$z_i = z_{f(i)}$ for all $n+1 \leq i \leq n+d$ and $z_j = z_{g(j)}$ for all $1 \leq j \leq n$. Then $\Zpoints^{f,g}_{n,k,d}$ is a $k$-dimensional linear subspace of $\QQ^{n+k+d}$. We have
\begin{equation}
\label{eq:irreducible-decomposition}
    \Zpoints_{n,k,d} = \bigcup_{f,g} \Zpoints_{n,k,d}^{f,g}
\end{equation}
so that $\Zpoints_{n,k,d}$ is an arrangement of $k$-dimensional subspaces in $(n+k+d)$-space and \eqref{eq:irreducible-decomposition} is the decomposition of $\Zpoints_{n,k,d}$ into its irreducible components.

For $f, g$ fixed, the coordinate ring of $\Zpoints_{n,k,d}^{f,g}$ is given by 
\begin{equation} 
    \QQ[\Zpoints_{n,k,d}^{f,g}] = \QQ[\xxx_n,\yyy_d,\ttt_k] / (x_j - y_{g(j)}, y_i - t_{f(i)} )
\end{equation}
and we have an evaluation isomorphism
\begin{equation} 
    \varepsilon_{f,g}: \QQ[\Zpoints_{n,k,d}^{f,g}] = \QQ[\xxx_n,\yyy_d,\ttt_k] / (x_j - y_{g(j)}, y_i - t_{f(i)} ) \xrightarrow{ \, \, \sim \, \, } \QQ[\ttt_k]
\end{equation}
determined by $x_j \mapsto t_{g(f(j))}, y_i \mapsto t_{f(i)}, t_k \mapsto t_k$. Each inclusion $\Zpoints_{n,k,d}^{f,g} \hookrightarrow \Zpoints_{n,k,d}$ induces a restriction map $\iota_{f,g}: \QQ[\Zpoints_{n,k,d}] \rightarrow \QQ[\Zpoints_{n,k,d}^{f,g}]$ between coordinate rings. Patching these maps together yields a $\QQ$-algebra map
\begin{equation}
    \iota: \QQ[\Zpoints_{n,k,d}] \xrightarrow{ \, \, \oplus \, \iota_{f,g} \, \, } \bigoplus_{f,g} \QQ[\Zpoints_{n,k,d}^{f,g}] \xrightarrow[ \, ,\ \sim \, \, ]{\, \, \oplus \, \varepsilon_{f,g} \, \, } \bigoplus_{f,g} \QQ[\ttt_k].
\end{equation}

\begin{lemma}
    \label{lem:iota-is-injective}
    The map $\iota$ is injective.
\end{lemma}

\begin{proof}
    This is true due to the decomposition in \eqref{eq:irreducible-decomposition} and because each $\varepsilon_{f,g}$ is an isomorphism.
\end{proof}

Proposition~\ref{prop:t-ideal-equality} gives the identification 
\begin{equation}\QQ[\Zpoints_{n,k,d}] = \QQ[\xxx_n,\yyy_d,\ttt_k]/\II(\Zpoints_{n,k,d}) = \QQ[\xxx_n,\yyy_d,\ttt_k]/J_{n,k,d}^{\QQ,\ttt}.\end{equation}
Taking $\symm_d$-invariants, Proposition~\ref{prop:rational-cohomology-presentation} implies that 
\begin{equation}
    \QQ[\Zpoints_{n,k,d}]^{\symm_d} = \QQ[\xxx_n,\yyy_d,\ttt_k]^{\symm_d}/(J_{n,k,d}^{\QQ,\ttt} \cap \QQ[\xxx_n,\yyy_d,\ttt_k]^{\symm_d}) = H^*_T(X_{n,k,d};\QQ).
\end{equation}
The $T$-fixed points $X_{n,k,d}^T$ correspond to words $w \in \WWW_{n,k,d}$: the fixed point corresponding to $w$ is the tuple of lines $(\langle \eee_{w(1)} \rangle, \cdots, \langle \eee_{w(n)} \rangle)$ where $\langle \eee_j \rangle$ is the line in $\CC^k$ spanned by the $j^{th}$ standard basis vector $\eee_j$.

\begin{proposition}
    \label{prop:restriction-injective}
    Let $i: X_{n,k,d}^T \hookrightarrow X_{n,k,d}$ be the inclusion. The induced cohomology map $$i^*: H^*_T(X_{n,k,d};\QQ) \rightarrow H^*_T(X_{n,k,d}^T;\QQ)$$ is injective. If we write
    $$H^*_T(X_{n,k,d}^T;\QQ) = \bigoplus_{w \, \in \, \WWW_{n,k,d}} \QQ[\ttt_k]$$
    with
    respect to the presentation of $H^*_T(X_{n,k,d};\QQ) = \QQ[\xxx_n,\yyy_d,\ttt_k]^{\symm_d}/(J_{n,k,d}^{\QQ,\ttt} \cap \QQ[\xxx_n,\yyy_d,\ttt_k]^{\symm_d})$ in Proposition~\ref{prop:rational-cohomology-presentation}, the map $\iota^*$ sends 
    $$x_i \mapsto t_{w(i)}, \quad  e_r(\yyy_d) \mapsto e_r(t_j \,:\, j \in w([n])), \quad t_i \mapsto t_i$$
    where the image of $e_r(\yyy_d)$ is the elementary symmetric polynomial of degree $r$ evaluated at the set of variables $t_j$ for which $j$ appears in the image of $w$.
\end{proposition}

\begin{proof}
    This follows from Lemma~\ref{lem:iota-is-injective} because restricting the domain of an injection to its $\symm_d$-invariant part yields an injection.
\end{proof}

An important feature of GKM theory is a combinatorial description of the image of the map $i^*: H^*_T(X;\QQ) \hookrightarrow H^*_T(X^T;\QQ) \cong \bigoplus_{x \in X^T} \QQ[\ttt_k]$ in terms of the 1-dimensional orbits of $T$ acting on $X$. Each of these orbits is homeomorphic to $\PP^1$ and connects two fixed points $x, x' \in X^T$. These 1-dimensional orbits form the edges in a graph on the vertices $X^T$ called the {\em GKM graph}; there is a at most 1-dimensional orbit connecting any two fixed points. The edges of the GKM graph correspond to divisibility conditions between components which characterize the image of $i^*$.

The $1$-dimensional $T$-orbits in $X_{n,k,d}$ (or even in $X_{n,k}$) do not have the structure described in the above paragraph. This failure happens in two different ways.

\begin{example}
\label{ex:gkm-failure-exit}
    Consider the space $X_{4,2}$ of spanning quadruples $(\ell_1, \dots, \ell_4)$ of lines in $\CC^2$ with the action of $T = (\CC^*)^2$. We represent quadruples of lines in $\CC^2$ by $2 \times 4$ matrices with no zero columns, up to the column-scaling action of $(\CC^*)^4$. The torus $T$ acts by row-scaling. We have a $1$-dimensional $T$-orbit 
    $$\left\{ \begin{pmatrix}  0 & 0 & 0 & 1 \\ 1 & 1 & 1 & z \end{pmatrix} \,:\, z \in \CC^*   \right\}$$
    whose closure within $\overline{X_{4,2}} = (\PP^1)^4$ connects the $T$-fixed points
    $$\begin{pmatrix} 0 & 0 & 0 & 1 \\ 1 & 1 & 1 & 0 \end{pmatrix} \quad \text{and} \quad \begin{pmatrix} 0 & 0 & 0 & 0 \\ 1 & 1 & 1 & 1 \end{pmatrix}.$$
    The fixed point on the right does not lie in $X_{4,2}$.
\end{example}

The behavior in Example~\ref{ex:gkm-failure-exit} stems from the noncompactness of spanning line moduli spaces. The next example gives another kind of geometric pathology.

\begin{example}
    \label{ex:gkm-failure-size}
    Continuing to work in the space $X_{4,2}$, for any $a \in \CC^*$ we have a $1$-dimensional $T$-orbit
    $$\left\{ \begin{pmatrix}  z & az & 0 & 1 \\ 1 & 1 & 1 & 0 \end{pmatrix} \,:\, z \in \CC^*   \right\}.$$
    Each of these orbits is distinct and connects the fixed points 
    $$\begin{pmatrix} 0 & 0 & 0 & 1 \\ 1 & 1 & 1 & 0 \end{pmatrix} \quad \text{and} \quad \begin{pmatrix} 1 & 1 & 0 & 1 \\ 0 & 0 & 1 & 0 \end{pmatrix}.$$
    These fixed points are therefore connected by more than one 1-dimensional $T$-orbit.
\end{example}

Examples~\ref{ex:gkm-failure-exit} and \ref{ex:gkm-failure-size} notwithstanding, Proposition~\ref{prop:restriction-injective} gives an explicit description of the restriction map $H^*_T(X_{n,k,d};\QQ) \to H^*_T(X^T_{n,k,d};\QQ)$. We have the following combinatorial problem.

\begin{problem}
    \label{prob:gkm-image}
    Characterize the image of the map $i^*$ in Proposition~\ref{prop:restriction-injective}.
\end{problem}

As a partial result towards Problem~\ref{prob:gkm-image}, suppose  $(f_w) \in \im(i^*)$ where $w$ ranges over maps $w: [n] \to [k]$ whose image has size $d$ and $f_w \in \QQ[\ttt_k]$ for all $w$. Let $w_1, w_2: [n] \to [k]$ be two such maps and suppose there exists a subset $I \subseteq [n]$ and distinct entries $j_1,j_2 \in [k]$ such that
\begin{quote}
    $(\star)$ for all $i \in I$ we have $w_1(i) = j_1$ and $w_2(i) = j_2$, and for all $i' \notin I$ we have $w_1(i') = w_2(i')$.
\end{quote}
We claim that $(t_{i_1} - t_{i_2}) \mid (f_{w_1} - f_{w_2})$. Indeed, since $(f_w) \in \im(i^*)$, there exists $g \in \CC[\xxx_n,\yyy_d,\ttt_k]^{\symm_d}$ such that for each $w$, we have $g \mapsto f_w$ under the map $x_i \mapsto t_{w(i)}, e_d(\yyy_d) \mapsto e_d(\ttt_k)$. The condition $(\star)$ implies that $f_{w_1} - f_{w_2}$ vanishes at $t_1 = t_2$, and the divisibility $(t_{i_1} - t_{i_2}) \mid (f_{w_1} - f_{w_2})$ follows.

The divisibility in the above paragraph has a geometric interpretation. For any $w \in [k]^n$, let $\eee_w$ be the $k \times n$ matrix with 1's in positions $(w(i),i)$ for $i= 1, \dots, n$ and 0's elsewhere. These matrices represent the $T$-fixed points $[\eee_w]$ in $(\PP^{k-1})^n$. If $w_1,w_2 \in [k]^d$ satisfy $(\star)$, let $\eee_{w_1,w_2}$ be the $k \times n$ matrix such that
\begin{equation}
    \text{column $i$ of $\eee_{w_1,w_2}$} = \begin{cases}
        \text{column $i$ of $\eee_{w_2}$} & \text{if $i \in I$,} \\
        0 & \text{otherwise.}
    \end{cases}
\end{equation}
Then $\{ [\eee_{w_1} + a \eee_{w_1,w_2}] \,:\, a \in \CC^* \}$ forms a $1$-dimensional $T$-orbit. The closure of this orbit
$$P_{w_1,w_2} := \overline{\{ [\eee_{w_1} + a \eee_{w_1,w_2}] \,:\, a \in \CC^* \}} = \{ [\eee_{w_1} + a \eee_{w_1,w_2}] \,:\, a \in \CC^* \} \cup \{ [\eee_{w_1}], [\eee_{w_2}] \}$$
is a copy of $\PP^1$, and for $g \in H^*_T(P_{w_1,w_2};\QQ)$ we have $(t_{j_1} - t_{j_2}) \mid (g(w_1) - g(w_2))$. (Here $g(w) \in H^*_T([\eee_w];\QQ) = \QQ[\ttt_k]$ is the pullback of $g$ along $\{[\eee_w]\} \hookrightarrow P_{w_1,w_2}.)$ By considering the inclusions $\{ [\eee_{w_i}] \} \hookrightarrow P_{w_1,w_2} \hookrightarrow X_{n,k,d}$ for $i=1,2$, the divisibility $(t_{i_1} - t_{i_2}) \mid (f_{w_1} - f_{w_2})$ of the last paragraph follows.

\section{Conclusion}
\label{sec:Conclusion}

In this paper we used orbit harmonics to calculate the torus equivariant cohomology of a moduli space $X_{n,k,d}$ of spanning line configurations which generalizes the spanning line configuration space $X_{n,k}$ introduced by Pawlowski and Rhoades \cite{PR}. If $X$ is a space with a torus action whose cohomology ring is an orbit harmonics quotient, one can ask for an analogous computation for its equivariant cohomology by turning the orbit harmonics parameters into torus variables. Several possible spaces $X$ on which to apply this program are as follows.

\begin{itemize}
    \item Let $V \subseteq \Gr(d,\CC^k)$ be a smooth Schubert variety. Let $X \subseteq X_{n,k,d}$ be the subvariety of tuples $(\ell_1, \dots, \ell_n)$ of lines in $\CC^k$ for which $\ell_1 + \cdots + \ell_n \in V$.
    \item Let $k$ be a positive integer and let $\alpha = (\alpha_1, \dots, \alpha_d)$ be integers with $1 \leq \alpha_i \leq k$ for all $i$. Rhoades used orbit harmonics to present \cite{RhoadesSpanning} the ordinary cohomology of the moduli space $X_{\alpha,k}$ of sequences $W_\bullet = (W_1, \dots, W_d)$ of linear subspaces $W_i \subseteq \CC^k$ with $\dim W_i = \alpha_i$ such that $W_1 + \cdots + W_d = \CC^k$. The rank $k$ torus acts on this space.
    \item Let $n \geq k \geq r$ be positive integers. Rhoades and Wilson used orbit harmonics to present \cite{RW} the cohomology of the space of $n$-tuples $\ell_\bullet = (\ell_1, \dots, \ell_n)$ of lines in $\CC^k$ such that $\ell_1 + \cdots + \ell_n = \CC^k$ and the sum $\ell_1 \oplus \cdots \oplus \ell_r$ is direct. The rank $k$ torus also acts on this space.
    \item Griffin, Levinson, and Woo \cite{GLW} defined and presented the cohomology of a family of varieties $Y_{n,\lambda,s}$ which generalize the cohomology rings of Springer fibers studied by Garsia and Procesi \cite{GP}. The presentation of $H^*(Y_{n,\lambda,s};\ZZ)$ in \cite{GLW}  arises from an orbit harmonics quotient $R_{n,\lambda,s}$ originally defined by Griffin \cite{Griffin}.
\end{itemize}
The Garsia-Procesi quotient rings $R_\lambda$ presenting the cohomology of Springer fibers $\mathcal{B}_\lambda$  are a special case of the final bullet point \cite{GP}. The rings $R_\lambda$ have an orbit harmonics interpretation; Kumar and Procesi presented \cite{KP} the torus-equivariant cohomology of $\mathcal{B}_\lambda$ as regular functions on the family of these orbit harmonics loci. Freitas and Mukhin \cite{FM} defined explicit polynomial ring quotients which could be relevant in presenting the Kumar-Procesi rings.

In the context of $X_{n,k,d}$, Lemma~\ref{lem:affine-paving} gives a `geometric' $\ZZ[\ttt_k]$-basis $[ \overline{C}_w ]_T \in H^*_T(X_{n,k,d};\ZZ)$ indexed by words $w \in \WWW_{n,k,d}$. One can ask for an algebraic interpretation of this basis.

\begin{problem}
    \label{prob:representatives}
    For $w \in \WWW_{n,k,d}$, find a polynomial representative for $[ \overline{C}_w ]_T$ under the presentation 
    $$H^*_T(X_{n,k,d};\ZZ) = \ZZ[\xxx_n,\yyy_d,\ttt_k]^{\symm_d}/I_{n,k,d}$$
    of Theorem~\ref{thm:borel-presentation}.
\end{problem}

When $k = d$, a solution to Problem~\ref{prob:representatives} is contained in the work of Pawlowski-Rhoades \cite{PR}. A word $w: [n] \to [k]$ is {\em convex} if it contains no subword of the form $i \dots j \dots i$ for $i \neq j$. The {\em convexification} $\mathrm{conv}(w)$ is the unique convex word with the same letter multiplicities as $w$ in which the initial letter values appear in the same order. We let $\sigma(w) \in \symm_n$ be the unique permutation with the fewest number of inversions which sorts $\conv(w)$ into $w$ If $w \in \WWW_{n,k}$ is a convex Fubini word, the {\em standardization} $\std(w)$ is the permutation in $\symm_n$ whose one-line notation is obtained from $w$ by replacing non-initial letters with $k+1, k+2, \dots, n$ from left to right.
For example, if $w = [2,4,1,1,2,4,1,3,4]$ we have %
\begin{multline*}\mathrm{conv}(w) = [2,2,4,4,4,1,1,1,3], \quad 
\sigma(w) = [1,5,2,6,9,3,4,7,8], \\ \text{ and } \quad \std(\mathrm{conv}(w)) = [2,5,4,6,7,1,8,9,3].\end{multline*}

 Let $w \in \WWW_{n,k}$ be a Fubini word. We have the corresponding $T$-stable cell $C_w \subseteq X_{n,k}$. We claim that the class $[\overline{C_w}]_T \in H^*_T(X_{n,k};\ZZ)$ is represented by 
\begin{equation}
\label{eqn:representative}
[\overline{C_w}]_T = \sigma(w)^{-1} \cdot \mathfrak{S}_{\std(\conv(w))}(-\mathbf{x}_n|\mathbf{t}_k)
\end{equation}
under the identification of Theorem~\ref{thm:borel-presentation}.
The notation in \eqref{eqn:representative} is as follows.
\begin{itemize}
    \item $\mathfrak{S}_{\std(\conv(w))}(-\mathbf{x}_n|\mathbf{t}_k)$ is the double Schubert polynomial indexed by the permutation $\std(\conv(w)) \in \symm_n$ evaluated at the variables $-x_1, - x_2, \dots, - x_n$ and $ t_1,  t_2, \dots, t_k$, and
    \item the permutation $\sigma(w)^{-1}$ acts on $\mathfrak{S}_{\std(\conv(w))}(-\mathbf{x}_n|\mathbf{t}_k)$ by permuting the $x$-variables.
\end{itemize}
To justify \eqref{eqn:representative}, one first observes that the natural action of $\symm_n$ on $X_{n,k}$ corresponds to permutation of the $x$-variables under the presentation of Theorem~\ref{thm:borel-presentation}. On the other hand, as in \cite[Proof of Prop. 5.11]{PR}, we have the equality
\begin{equation}
    C_w = C_{\sigma(w) \cdot \conv(w)} = C_{\conv(w)} \cdot \sigma^{-1}(w)
\end{equation}
of subspaces of $X_{n,k}$. This reduces us to the case where $w$ is convex, so $w = \conv(w)$. When $w$ is convex, it is shown in \cite[Proof of Lem. 5.9]{PR} that $\Omega_w := \overline{C_w}$ is a certain degeneracy locus (see e.g. \cite{PR} or \cite{AF} for more details). Since $\Omega_w$ is $T$-stable, \cite[Cpt. 11, Cor. 6.5]{AF} applies to give \eqref{eqn:representative}. 
The authors do not know a solution to Problem~\ref{prob:representatives} for $d < k$.

\section{Acknowledgements}

The authors are grateful to Sean Griffin, Ed Richmond, and Travis Scrimshaw for helpful conversations. B. Rhoades was partially supported by NSF Grant DMS-2246846. T. Matsumura was supported partially by JSPS Grant-in-Aid for Scientific Research (C) 20K03571 and (B) 23K25772.

\end{document}